\documentclass[12pt]{article}
\usepackage{amsmath, amsfonts, amsthm}
\usepackage{xcolor}
\usepackage{tikz}
\usepackage{tkz-euclide}   
\usetikzlibrary{shapes.misc}
\usepackage{enumitem}
\tikzset{cross/.style={cross out, draw=black, minimum size=2*(#1-\pgflinewidth), inner sep=0pt, outer sep=0pt},
cross/.default={1pt}}
\numberwithin{equation}{section}

\usepackage{geometry}
\newtheorem{theorem}{Theorem}[section]

\newtheorem{lemma}[theorem]{Lemma}

\newtheorem{proposition}[theorem]{Proposition}
\renewenvironment{proof}[1][Proof]{\noindent\textbf{#1.} }{\ \rule{0.5em}{0.5em}}
\theoremstyle{definition}
\newtheorem{definition}{Definition}[section]
\newtheorem{remark}[definition]{Remark}
\newtheorem{example}[definition]{Example}

\def\defn#1{{\bf\itshape #1}}

\definecolor{cadmiumgreen}{rgb}{0.0, 0.42, 0.24}
  \definecolor{burgundy}{rgb}{0.5, 0.0, 0.13}
\usepackage{hyperref}
\hypersetup{
    colorlinks=true,
    linkcolor=burgundy,
    filecolor=magenta,      
    urlcolor=burgundy,
    citecolor=cadmiumgreen
}

\usepackage{graphicx,wrapfig}
\usepackage[font={small,sf, color=black!80}]{caption}

\newcommand{\E}{E_{\mbox{\tiny red}}^{n+N}}
\begin{document}

\title{Braids of the $N$-body problem II: carousel solutions \\by cabling central configurations}
\author{M. Fontaine and C. Garc\'ia-Azpeitia}
\maketitle
\begin{center}
\begin{minipage}[]{14cm}
	\small \textbf{Abstract.} We prove the existence of relative periodic solutions of the planar $N=\sum_{j=1}^n k_j$-body problem starting with $n$ bodies moving close to a non-degenerate central configuration and replacing each of them with clusters of $k_j$ bodies that move close to a small central configuration. We name these solutions carousel solutions. The proof relies on blow-up techniques for variational methods used in our previous work \cite{Braids}.
	
\vspace{0.2cm}

\vspace{0.2cm}
\noindent \textbf{Keywords.} $N$-body problem, periodic solutions, perturbation theory.
\end{minipage} 
\end{center}
\section{Introduction}

The existence of braids in the $N$-body problem has been intensively studied
since the pioneering work of Poincar{\' e}. In the case of strong forces,
the classical approach exploits the fact that the Euler action functional
blows up at any orbit belonging to the boundary of a braid class, which
implies the existence of minimizers for tied braid classes by using the
direct method of the calculus of variations, see \cite{Go2}, \cite{Mon1} and
references therein. Later on, C. Moore in \cite{Mo93} found braids for
gravitational and weak forces by making numerical continuations of the
aforementioned braids for strong forces.

For gravitational forces, the existence of braids in the $3$-body problem ($%
N=2+1$) was established in the classical setting of the Sun-Earth-Moon
system. These solutions are obtained by replacing one body with two in a
circular motion of the $2$-body problem. There is a large literature related
to this problem \cite{4}, \cite{8}, \cite{2}, \cite{3}, \cite{Con}. An
extension of these solutions to the $4$-body problem ($N=3+1$), where one
body in the Lagrange triangular configuration is replaced by two, was
obtained in \cite{Ch}. This result was then extended in \cite{Meyer} to the
general case $N=n+1$, where one body in a nondegenerate central
configuration is replaced by two, making use of symplectic scalings and the
Implicit Function Theorem.

In our previous work \cite{Braids}, we established a new approach based on
blow-up techniques to construct solutions of the $N=n+1$-body problem in $E=%
\mathbb{R}^{2d}$ for central forces (weak, gravitational and strong), where
one body in a nondegenerate central configuration of $n$ bodies is replaced
by two. Along with such a solution, two of the bodies rotate uniformly around their
center of mass. The other $n-1$ bodies and the center of mass of the pair
remain, at each time, close to a central configuration. When $d=1$, these
solutions are \defn{braids} obtained by replacing a strand in a braid by
a new braid, a process that was called \defn{cabling} by C. Moore in \cite{Mo93}.

In the present article, we generalize the construction in \cite{Braids} to
replace several bodies in a non-degenerate central configuration by clusters
of bodies arranged themselves in small central configurations. We name these
new solutions \defn{carousel solutions} (Figure \ref{fig:carousel}). Due to
the higher level of complexity that involves dealing with multiple clusters,
we only treat the case when the motion takes place in the plane ($d=1$). One
may consider symmetry constraints as in \cite{Braids} to possibly extend these
results to higher even-dimensional spaces.

\begin{figure}[th]
\centering
\begin{minipage}{.76\textwidth}
\centering
\includegraphics[scale=0.25]{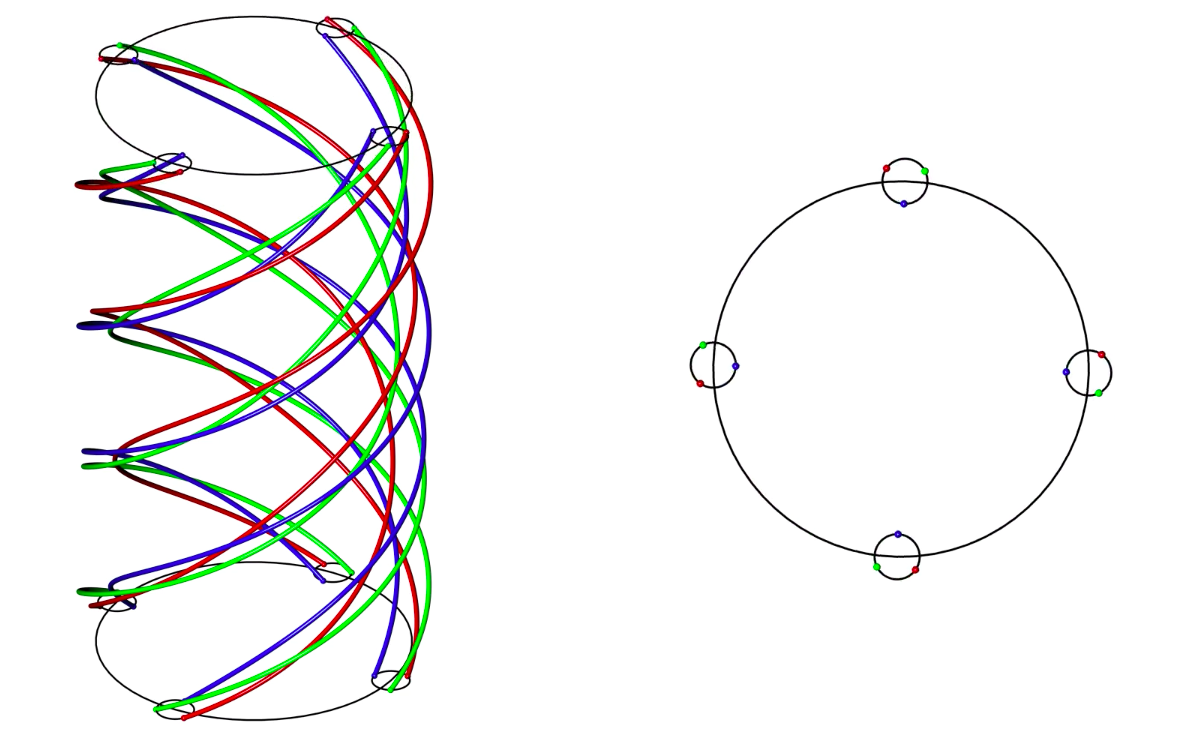}
\caption{Carousel solutions of the $12$-body problem. The four bodies in a polygonal configuration are replaced by clusters of three bodies of equal masses arranged at the vertices of an equilateral triangle (Lagrange triangular configuration).}
\label{fig:carousel}
\end{minipage}
\end{figure}

\subsection*{Model: the $N$-body and $N$-vortex filament problems}

We first introduce a multi-index notation to describe the positions of the
bodies in their respective cluster. The positions are given by vectors $%
q_{j,k}$ in $E=\mathbb{R}^{2}$ for $k=1,\dots,k_{j}$ and $j=1,\dots,n$. To
each $q_{j,k}$ we attach a positive mass $m_{j,k}>0$. The index $j$
represents the cluster that contains $k_{j}$ bodies. Setting $%
N=\sum_{j=1}^{n}k_{j}$, the equations of motion of the $N$-body problem are 
\begin{equation}
m_{j,k}\ddot{q}_{j,k}=-\sum_{\substack{ (j^{\prime},k^{\prime}) \\ \left(
j^{\prime},k^{\prime}\right) \neq\left( j,k\right) }}m_{j,k}m_{j^{\prime
},k^{\prime}}\frac{q_{j,k}-q_{j^{\prime},k^{\prime}}}{\Vert
q_{j,k}-q_{j^{\prime},k^{\prime}}\Vert^{\alpha+1}}\qquad
k=1,\dots,k_{j}\qquad j=1,\dots,n.   \label{NBP}
\end{equation}
Without loss of generality, we suppose that $k_{j}>1$ for $j=1,\dots,n_{0}$
and $k_{j}=1$ for $j=n_{0}+1,\dots,n$. That is, the $j$-cluster contains
only one body when $j=n_{0}+1,\dots,n$.

The relevant cases from the physical point of view are the gravitational
potential ($\alpha=2$) and the logarithmic potential ($\alpha=1$). In the
latter case, equations \eqref{NBP} govern the approximate interaction of $N$
steady vortex filaments in fluids (Euler equation) \cite{KPV03}, \cite%
{BaMi12}, \cite{KMD95}, \cite{GaCr15}, Bose-Einstein condensates
(Gross--Pitaevskii equation) and superconductors (Ginzburg-Landau equation) 
\cite{DelPK08}, \cite{ContrerasJerrard16}. Although  it is worth mentioning that there are some subtle differences
in the equations for steady vortex filaments with respect to equations (\ref%
{NBP}. Namely, the role of the masses in the
models of filament are played by the circulations, which can be negative,
quantized or weighted by different factors representing the vortex core of
filaments.

\subsection*{Main result: existence of carousel solutions}

We construct carousel solutions starting from multiple central
configurations. For $j=1,\dots,n_{0}$ the $j$-cluster is close to a %
\defn{central configuration} of $k_{j}$-bodies $a_{j}=(a_{j,1},\dots
,a_{j,k_{j}})$ such that 
\begin{equation}
m_{j,k}a_{j,k}=\sum_{\substack{ k^{\prime}=1 \\ k^{\prime}\neq k}}%
^{k_{j}}m_{j,k}m_{j,k^{\prime}}\frac{a_{j,k}-a_{j,k^{\prime}}}{\left\Vert
a_{j,k}-a_{j,k^{\prime}}\right\Vert ^{\alpha+1}}\qquad k=1,\dots,k_{j}~. 
\label{ccj}
\end{equation}
We assume that each central configuration $a_{j}$ has zero center of mass.
For $j=n_{0}+1,\dots,n$ we assume that the $j$-clusters are made of a single
body.
\begin{figure}[th]
\centering
\includegraphics[scale=0.7]{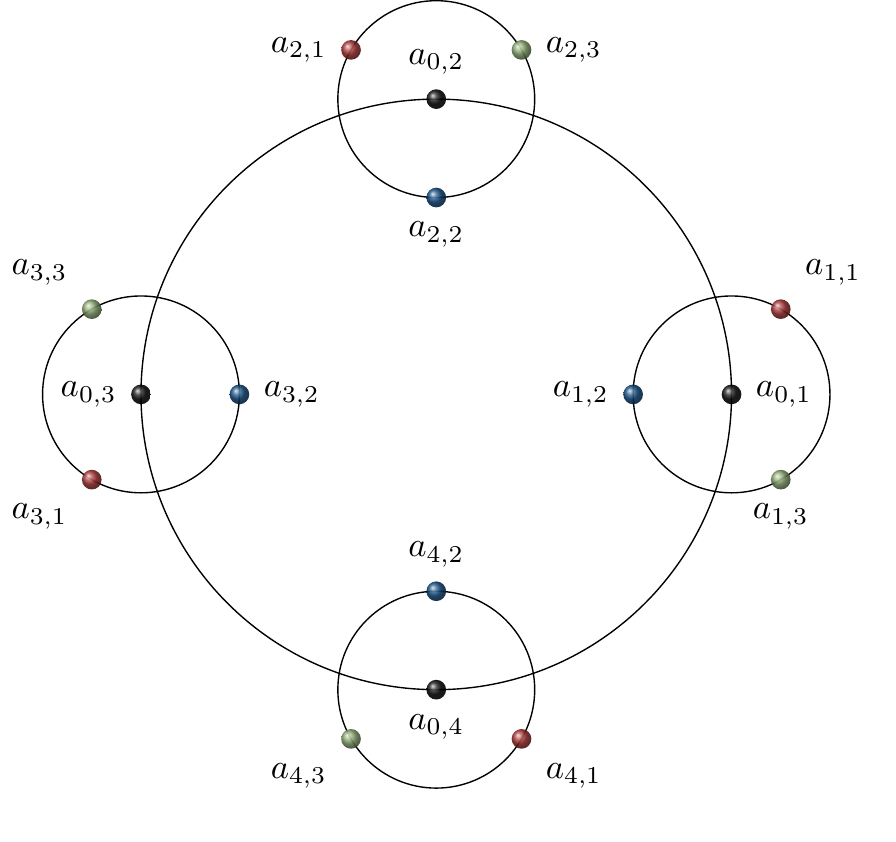}
\end{figure}
Our main result states that we can replace $n_{0}$ bodies in a central
configuration of $n$-bodies by clusters of $k_{j}$-bodies close to the
central configurations $a_{j}$. For this purpose, we start from a central
configuration $a_{0}=(a_{0,1},\dots,a_{0,n})$ with masses $%
M_{j}=\sum_{k=1}^{k_{j}}m_{j,k}$ given by the total masses of the $j$%
-clusters for $j=1,...,n_0$. This means that 
\begin{equation}
M_{j}a_{0,j}=\sum_{\substack{ j^{\prime}=1 \\ j^{\prime}\neq j}}%
^{n}M_{j}M_{j^{\prime}}\frac{a_{0,j}-a_{0,j^{\prime}}}{\left\Vert
a_{0,j}-a_{0,j^{\prime}}\right\Vert ^{\alpha+1}}\qquad j=1,\dots,n_0. 
\label{cc0}
\end{equation}
We also assume, without loss of generality, that the central configuration $%
a_{0}$ has zero center of mass.

In Theorem \ref{main result}, we construct the \defn{carousel solutions} of %
\eqref{NBP} starting from $n_{0}$ central configurations $%
a_{1},\dots,a_{n_{0}}$ satisfying \eqref{ccj} and $a_{0}$ satisfying %
\eqref{cc0}. We suppose that $a_{j}$ are $2\pi p_{j}$-nondegenerate for $p_{1},\dots,p_{n_{0}}\in\mathbb{Z}%
\backslash\{0\}$ fixed integers (Definition \ref{Def1}),
and also that $a_{0}$ is nondegenerate (Definition \ref%
{Def3}). We then prove that, for every sufficiently small $%
\varepsilon$, there are at least $n_{0}+1$ solutions of the $N$-body problem %
\eqref{NBP} with components of the form%
\begin{align}
q_{j,k}(t) & =\exp(tJ)u_{0,j}(\nu t)+r_{j}\exp(t\omega_{j}J)u_{j,k}(\nu
t),\quad j=1,\dots,n_{0},\quad k=1,\dots,k_{j}  \label{carousel solutions} \\
q_{j,1}(t) & =\exp(tJ)u_{0,j}(\nu t),\quad\;j=n_{0}+1,\dots,n.  \notag
\end{align}
The matrix $J$ is a complex structure on $E$. The paths $(u_{0,1}(s),%
\dots,u_{0,n}(s))$ and $(u_{j,1}(s),\dots,u_{j,k_{j}}(s))$ remain $%
\varepsilon$-close in a space of $2\pi$-periodic paths to the central
configurations $a_{0}$ and $a_{j}$ for $j=1,...,n_{0}$ respectively. The
amplitudes and the frequencies of rotation of the clusters 
\begin{equation*}
r_{j}=\left( 1+p_{j}\nu\right) ^{-2/(\alpha+1)}\qquad\mbox{and}\qquad
\omega_{j}=1+p_{j}\nu 
\end{equation*}
are controlled uniformly by the parameter $\varepsilon$, where 
\begin{equation*}
\nu=\varepsilon^{-\left( \alpha+1\right) /2}-1 
\end{equation*}
represents the frequency of the perturbation from the arrangement of the
central configurations in rigid motion. Furthermore, we show that the solutions are
determined by different orientating phases $\vartheta_{j}\in S^{1}$ such
that $u_{j,k}(s)=\exp(\vartheta_{j}J)a_{j,k}+\mathcal{O}(\varepsilon)$ for $%
j=1,...,n_{0}$. Thus the result can be rephrased as 
\begin{align}
q_{j,k}(t) & =\exp(tJ)a_{0,j}+r_{j}\exp((t\omega_{j}+\vartheta_{j})J)a_{j,k}+%
\mathcal{O}(\varepsilon)\ \quad j=1,\dots,n_{0},\quad k=1,\dots,k_{j} \\
q_{j,1}(t) & =\exp(tJ)a_{0,j}+\mathcal{O}(\varepsilon)\quad\;j=n_{0}+1,%
\dots,n,  \notag
\end{align}
where $\mathcal{O}(\varepsilon)$ are quasi-periodic functions of order $%
\varepsilon$. Thus the parameter $\varepsilon $ provides a uniform measure of the shrink of
each cluster because the amplitudes $r_{j}=\mathcal{O}(\varepsilon )$.

The solutions that we obtain are quasi-periodic if $\nu\notin\mathbb{Q}$. If
the frequency $\nu=p/q$ is rational, the frequencies $\omega_{j}=(q+p_{j}p)/q
$ are rational as well. Our theorem implies for this case that, for any fixed
integer $q>0$, there is some integer $p_{0}>0$ such that, for each integer $%
p>p_{0}$, the components $q_{j,k}(t)$ are $2\pi q$-periodic. These are %
\defn{braid solutions}, where the centers of mass of the $n$ clusters (close
to the central configuration $a_{0}$) wind around the origin $q$ times,
while the configuration of $k_{j}$ bodies in each cluster winds around its
center of mass $q+p_{j}p$ times (Figure \ref{fig:carousel}). The sign of $%
p_{j}$ determines whether the $j$-cluster has a \emph{prograde} or a
retrograde rotation with respect to the whole system. When the rotation is
prograde ($p_{j}>0$), the cluster rotates in the same direction as the main
relative equilibrium. When the rotation is \emph{retrograde} ($p_{j}<0$),
the cluster rotates in the opposite direction.

Our proof relies on the assumption that the central configurations $a_{j}$
for $j=1,\dots ,n_{0}$ are $2\pi p_{j}$-nondegenerate. In simple terms, the $%
2\pi p_{j}$-nondegeneracy condition means that the group orbit of $a_{j}$ is
a nondegenerate critical manifold for the action functional of the $k_{j}$-body problem
defined in the space of $2\pi p_{j}$-periodic paths (Definition \ref{Def1}).
In Section 4 we prove that this condition holds for an infinite number of $%
k_{j}$-polygonal configurations with equal masses when $\alpha \in
(1,1+\delta )$ for a positive small $\delta $. In the case $\alpha \geq 1$
(but $\alpha \neq 2$) we show that this condition holds for the Lagrange
triangular configuration with different masses provided they satisfy the
relation 
\begin{equation}
\beta :=27\frac{m_{1}m_{2}+m_{1}m_{3}+m_{2}m_{3}}{\left(
m_{1}+m_{2}+m_{3}\right) ^{2}}>9\left( \frac{3-\alpha }{1+\alpha }\right)
^{2}.  \label{ineq}
\end{equation}%
We can therefore always replace any body in a nondegenerate central
configuration by regular $k_{j}$-polygons. We conjecture that the $2\pi p_{j}
$-nondegeneracy condition is generic for central configurations in the case $%
\alpha \geq 1$ ($\alpha \neq 2$).

When $\alpha=2$, all the central configurations are $2\pi p_{j}$-degenerate
due to the existence of the elliptic homographic solutions which are
generated by a central configuration. This implies that our result cannot be
directly extended to the case $\alpha=2$. In Theorem \ref{main result2} we
extend our result to the gravitational case under two additional
assumptions: (i) we can divide only a central body in a central
configuration which is symmetric under $2\pi/m$-rotations with $m\geq2$ and
(ii) the central configuration $a_{1}$ of the 1-cluster needs to be $2\pi/m$%
-nondegenerate (Definition \ref{Def2}). There are plenty of central
configurations satisfying the first assumption such as the Maxwell
configuration \cite{Ro00}, the nested polygonal configurations with a center
appearing in \cite{GaIz11} \cite{Mont15} \cite{Webs}. Regarding the second assumption, we prove in
Section 4 that the $2\pi/m$-nondegeneracy condition of a central
configuration holds for the $k_{1}$-polygonal configurations with equal
masses for $k_{1}=4,...,1000$. We show that this condition also holds true
in the case of the Lagrange triangular configuration when the masses satisfy
the inequality (\ref{ineq}) with $\alpha=2$. Therefore, we can always
replace the central body in a $\ $symmetric central configuration by a
regular $k_{1}$-polygon. We conjecture that the $2\pi/m$-nondegeneracy
condition is generic for central configurations.

\subsection*{Method: perturbation of nondegenerate critical manifolds}

Our method starts by writing down the Euler-Lagrange equations with respect
to the action functional $\mathcal{A}$ of the $N$-body problem. We implement
several changes of coordinates in configuration space that involve
Jacobi-like coordinates, rotating frames and scalings of the amplitudes of
the clusters. We extend the action functional $\mathcal{A}$ to new
coordinates 
\begin{equation*}
u=(u_{0},u_{1},...,u_{n})\in E^{n}\times E^{k_{1}}\times ...\times E^{k_{n}},
\end{equation*}%
where the action functional splits into two terms $\mathcal{A}(u)\mathcal{=A}%
_{0}(u)+\mathcal{H}(u)$, and the Euler-Lagrange equations of $\mathcal{A}_{0}(u)$
are uncoupled in the components $u_{j}$ and are given simply by the $k_{j}$%
-body problem in $u_{j}$ for $j=1,\dots ,n_{0}$ and the $n$-body problem for 
$u_{0}$. The action of the $N$-body problem corresponds to the restriction of 
$\mathcal{A}(u)$ to the subspaces $E_{j}$ defined in \eqref{EJ E0} as the subspace of zero center of mass of the $k_{j}$-body problem.
In particular, we have the convention that $u_{j}=0\in E_{j}$ for the
clusters $j=n_{0}+1,\dots ,n$. The \defn{coupling term} $\mathcal{H=O}%
(\varepsilon )$ encodes the interactions of the
different clusters in the new coordinates, which depends on $\varepsilon $ through the parameters $%
r_{j},\omega _{j},\nu $ introduced before and is small of order $\varepsilon 
$. 

The Euler-Lagrange equations of $\mathcal{A}$ are set as a gradient in a
subspace of $2\pi $-periodic paths denoted by $X$. We fix a configuration 
\begin{equation*}
u_{a}=(a_{0},a_{1},\dots ,a_{n})
\end{equation*}%
where $a_{0}$ is a central configuration \eqref{cc0} of the $n$-body
problem, $a_{j}$ is a central configuration \eqref{ccj} of the $k_{j}$-body
problem for $j=1,\dots ,n_{0}$ and, according to our convention, $a_{j}=0$
for $j=n_{0}+1,...,n$. It turns out that the configuration $u_{a}$ is a
critical point of $\mathcal{A}_{0}$. Furthermore, the functional $\mathcal{A}%
_{0}(u)$ is invariant under the action of the torus group $G=U(1)^{n+1}$ on $%
X$ defined by 
\begin{equation*}
(g_{0},g_{1},\dots ,g_{n})\cdot (u_{0},u_{1}\dots
,u_{n})=(g_{0}u_{0},g_{1}u_{1},\dots ,g_{n}u_{n}).
\end{equation*}%
where $(g_{0},g_{1},\dots ,g_{n})\in U(1)^{n+1}$. The action of $g_{0}$
rotates the $n$-body problem consisting of the $n$ centers of mass of the
clusters about the origin. The action of each $g_{j}$ rotates the
configuration of the $j$-cluster about its center of mass for $j=1,...,n_{0}$
and acts trivially on $u_{j}=0$ for $j=n_{0}+1,...,n$.

It follows that the group orbit $G(u_{a})$ is a critical manifold of $%
\mathcal{A}_{0}$. The perturbation $\mathcal{H=O}(\varepsilon )$ breaks the $%
G$-symmetry, in the sense that it is invariant under the action of the
diagonal subgroup $H=\widetilde{U(1)}$ of $G$. Our theorem is obtained by
proving the persistence of $H$-orbits of critical points of the perturbed
action $\mathcal{A}=\mathcal{A}_{0}+\mathcal{H}$ in a tubular neighbourhood
of the critical manifold $G(u_{a})$ in $X$. The core of the proof (Section $3$) relies on three Lyapunov-Schmidt
reductions: a reduction to finite dimension, a reduction that regularizes
the functional, and a reduction that uses Palais-slice coordinates near the
orbit $U(1)^{n}(a_{1},...,a_{n})$. We conclude that
finding critical points of $\mathcal{A}$ in a neighbourhood of $G(u_{a})$ is
equivalent to doing so for some regular function 
\begin{equation*}
\Psi _{\varepsilon }^{\prime }:U(1)^{n_{0}}\rightarrow \mathbb{R}
\end{equation*}
defined on the compact manifold $U(1)^{n_{0}}$ which represents the orbit of 
$(a_{1},...,a_{n})$ under the action of the group $U(1)^{n}$. A major
difference regarding our previous work \cite{Braids} is that to obtain $%
\Psi _{\varepsilon }^{\prime }$ we need to solve first the component $u_{0}$%
. The delicate part of the procedure is finding uniform estimates in $%
\varepsilon $ because the functional $\mathcal{A}_{0}$ explodes when $\varepsilon \rightarrow 0$  at different
scales. The theorem is easily obtained as
a consequence of the fact that the Lyusternik-Schnirelmann category of the
compact manifold $U(1)^{n_{0}}$ is $n_{0}+1$, which gives a lower bound for
the number of critical points of $\Psi _{\varepsilon }^{\prime }$. To each
of these critical points corresponds a different way of orienting the
central configurations $a_{j}$ for $j=1,...,n_{0}$ with respect to the
central configuration $a_{0}$. Generically, from Morse theory, the functional has in fact $2^{n_0}$ critical points.

It is worth mentioning that many authors have analyzed the persistence of
solutions near a nondegenerate critical manifold under perturbation. In our
case, the critical manifold of $\mathcal{A}_{0}$ is the group orbit $G(u_{a})
$ and the perturbation $\mathcal{H}$ is $H$-invariant. The perturbation of
nondegenerate critical manifolds consisting of group orbits has been studied
previously in \cite{Dan}, \cite{Van}, \cite{Lau}, \cite{Chos}, \cite{Am2}, 
\cite{Fo}, and references therein. In the context of the $n$-body problem,
the papers \cite{Mo} and \cite{Am1} analyze the breaking of symmetries of
the critical manifold of periodic solutions of the Kepler problem when a
non-radial external force is introduced. Our work uses ideas already
appearing in all those works, although a remarkable difference is that, in
our case, the functional $\mathcal{A}_{0}$ explodes as $\varepsilon
\rightarrow 0$ at different scales. We solve this problem by means of a
procedure that is similar to our previous work \cite{Braids} and was
motivated by the blow-up methods appearing in \cite{Ba16} and \cite{Ba17}.

Blow-up techniques have a long history. For instance, they were used by Floer
and Weinstein in \cite{Flo} to find single-soliton standing waves of the
nonlinear Schr{\" o}dinger equation in dimension one. The critical
manifold in that case is non-compact due to the action of the
group $\mathbb{R}$. Later on, these ideas were used to obtain multi-bump
solitons for Schr{\" o}dinger equations in higher dimension - the interested reader
can consult the large bibliography in \cite{Am3}. We think that it is possible
to use similar methods to extend our results to the gravitational case by
taking into consideration the non-compact critical manifold of all the elliptic
homographic solutions generated by the central configurations.

\section{Problem setting for carousels}

Let $E=\mathbb{R}^{2}$ with inner product $\langle \cdot ,\cdot \rangle $.
We consider $N=\sum_{j=1}^{n}k_{j}$ bodies moving in $E$ under the influence
of a central force field. We use a multi-index notation which simplifies
greatly the treatment of the problem. The positions of the bodies in $E$ are
denoted by $q_{j,k}$ where $k=1,\dots ,k_{j}$ and $j=1,\dots ,n.$ To each of
the positions we attach positive masses $m_{j,k}>0$. The index $j$
represents the cluster of bodies that contains $k_{j}$ bodies.

We define the kinetic energy and the potential function by 
\begin{eqnarray*}
K &=&\frac{1}{2}\sum_{j=1}^{n}\sum_{k=1}^{k_{j}}m_{j,k}\Vert \dot{q}%
_{j,k}\Vert ^{2} \\
U &=&\frac{1}{2}\sum_{\substack{ (j^{\prime },k^{\prime }),\left( j,k\right) 
\\ \left( j^{\prime },k^{\prime }\right) \neq \left( j,k\right) }}%
m_{j,k}m_{j^{\prime },k^{\prime }}\phi _{\alpha }(\Vert q_{j,k}-q_{j^{\prime
},k^{\prime }}\Vert )
\end{eqnarray*}%
where $\Vert \dot{q}_{j,k}\Vert ^{2}=\langle \dot{q}_{j,k},\dot{q}%
_{j,k}\rangle $ and $\phi _{\alpha }$ is a function such that $\phi _{\alpha
}^{\prime }(r)=-r^{-\alpha }$. The factor $1/2$ in the potential $U$ appears
due to the double sum of the same term. The Newtonian potential corresponds
to $\phi _{2}(r)=1/r$ and the vortex filament potential corresponds to $\phi
_{1}(r)=-\ln (r)$. Newton's laws of motion are: 
\begin{equation*}
m_{j,k}\ddot{q}_{j,k}=\nabla _{q_{j,k}}U=-\sum_{\substack{ (j^{\prime
},k^{\prime })  \\ \left( j^{\prime },k^{\prime }\right) \neq \left(
j,k\right) }}m_{j,k}m_{j^{\prime },k^{\prime }}\frac{q_{j,k}-q_{j^{\prime
},k^{\prime }}}{\Vert q_{j,k}-q_{j^{\prime },k^{\prime }}\Vert ^{\alpha +1}}%
,\qquad k=1,\dots ,k_{j},\qquad j=1,\dots ,n. 
\end{equation*}%
Let $\mathcal{L}=K+U$ be the Lagrangian of the system. The \textbf{\itshape %
action functional} 
\begin{equation*}
\mathcal{A}(q)=\int_{0}^{T}\mathcal{L}(q(t),\dot{q}(t))dt 
\end{equation*}%
is taken over the Sobolev space $H^{1}([0,T],E^{N})$ of paths $%
q:[0,T]\rightarrow E^{N}$ such that $q$ and its first derivative $\dot{q}$
are square integrable in the sense of distributions.

\subsection{Jacobi-like coordinates}

Similarly to \cite{Braids} we make a change of coordinates so that the
action functional splits into two terms $\mathcal{A}_{0}+\mathcal{H}$. The
Euler-Lagrange equations of $\mathcal{A}_{0}$ give rise to a set of
uncoupled $k_{j}$-body problems and an $n$-body problem. This procedure is a
generalization of \cite{Braids} which allows us to produce \textbf{\itshape %
carousel solutions} of the $N$-body problem. For this purpose we define new
coordinates 
\begin{equation}
q_{j,k}=Q_{0,j}+Q_{j,k},\quad k=1,\dots ,k_{j},\quad j=1,\dots ,n~
\label{Q coordinates}
\end{equation}%
which are overdetermined in the coordinates $Q_{0,j}$ and $%
Q_{j,k}$. To address this matter
we choose components $Q_{j,k}$ satisfying the $n$ constraints 
\begin{equation}
\sum_{k=1}^{k_{j}}m_{j,k}Q_{j,k}=0,\quad j=1,\dots ,n.  \label{Constraint}
\end{equation}

Let $Q=\left( Q_{0},Q_{1},\dots ,Q_{n}\right), $ where $Q_{0}=(Q_{0,1},\dots
,Q_{0,n}) \in E^{n}$ represent the positions of the
centers of mass of the $n$ clusters, and each $Q_{j}=(Q_{j,1},\dots
,Q_{j,k_{j}}) \in E^{k_{j}}$ represents the positions of the $k_{j}$ bodies
in the $j$-cluster. Denoting by $M_{j}=\sum_{k=1}^{k_{j}}m_{j,k}$ the total
mass of the $j$-cluster, we use \eqref{Q coordinates} and \eqref{Constraint}
to write the center of mass of the whole system as 
\begin{equation}
\sum_{j=1}^{n}\sum_{k=1}^{k_{j}}m_{j,k}q_{j,k}=\sum_{j=1}^{n}M_{j}Q_{0,j}%
\text{.}  \label{center of mass}
\end{equation}%
From now on, we impose the centers of mass of the $n$ clusters and the
center of mass of the $n$-body problem to be zero by considering the
coordinates such that 
\begin{equation*}
Q=\left( Q_{0},Q_{1},\dots ,Q_{n}\right) \in \E:=E_{0}\times E_{1}\times
\dots \times E_{n}, 
\end{equation*}%
where the subspaces $E_{j}$ are given by 
\begin{equation}  \label{EJ E0}
E_{0}:=\left\{ Q_{0}\in E^{n}:\sum_{j=1}^{n}M_{j}Q_{0,j}=0\right\} ,\quad
E_{j}:=\left\{ Q_{j}\in E^{k_{j}}:\sum_{k=1}^{k_{j}}m_{j,k}Q_{j,k}=0\right\},
\end{equation}%
for $j=1,\dots ,n~$. Since there is only one body in a $j$-cluster for $%
j=n_0+1,\dots ,n$, the constraint \eqref{Constraint} with $k_{j}=1$ implies
that $Q_{j}\in E_{j}=\{0\}$. Hence the position of a single body is
determined by $q_{j,1}=Q_{0,j}$ with $M_{j}=m_{j,1}$.

\begin{figure}[ht!]
\centering
\begin{minipage}{.85\textwidth}
\centering
\includegraphics[scale=0.65]{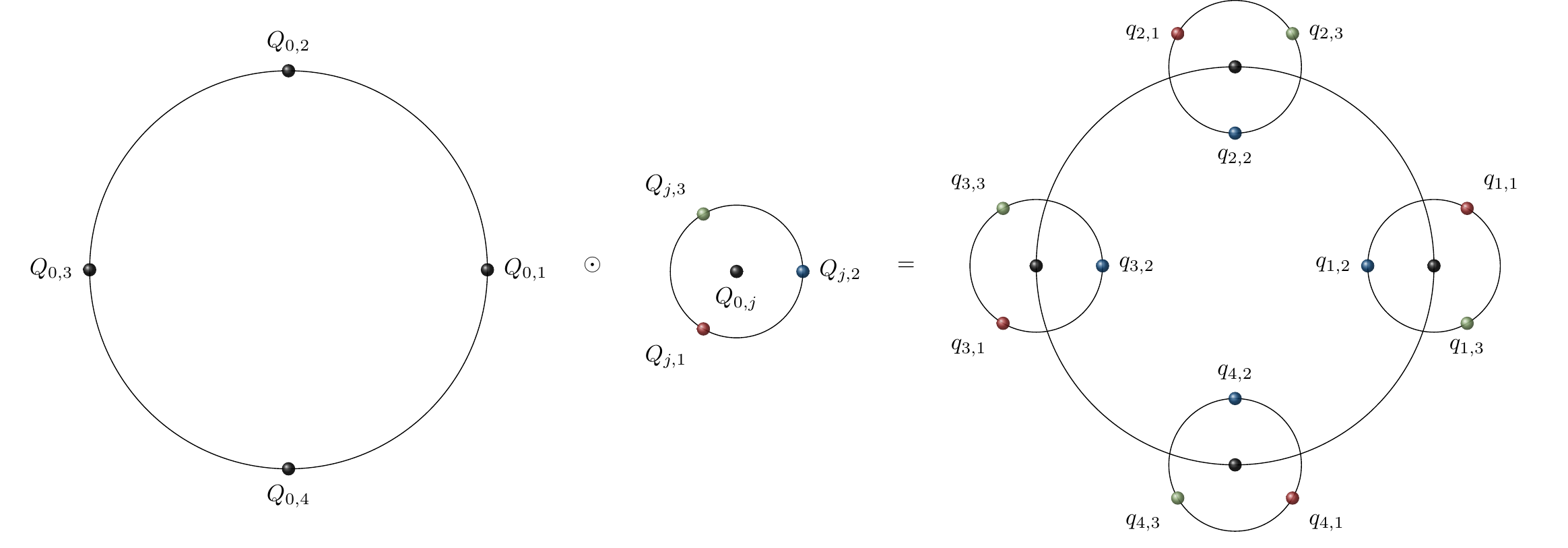}
\caption{The representation of the Jacobi-like coordinates with $q_{j,k}=Q_{0,j}+Q_{j,k}$ in the case $j=1,2,3,4$ and $k_j=1,2,3$ for each $j$. The symbol $\odot$ denotes the cabling operation of \cite{Mo}.}
\end{minipage}
\end{figure}

The Jacobi-like coordinates allow us to write the action functional as $%
\mathcal{A}(Q)=\mathcal{A}_{0}(Q)+\mathcal{H}(Q)$.\emph{\ }It is important
to remark that we consider this functional extended to $H^{1}([0,T],E^{N+n})$%
, and that the solutions to the $N$-body problem are the solutions of the
action $\mathcal{A}(Q)$\ with the holonomic constraint that $Q\in \E$\emph{. 
}The first term $\mathcal{A}_{0}$ is called the \defn{unperturbed functional}%
. The remaining part $\mathcal{H}$ is a \defn{coupling term} invariant by
linear isometries which becomes small when the norms $\Vert Q_{j,k}\Vert $
simultaneously become small for every $k,j$.

\begin{proposition}
\label{Prop: cabling coord} In Jacobi-like coordinates, the action
functional becomes%
\begin{equation*}
\mathcal{A}(Q)=\mathcal{A}_{0}(Q)+\mathcal{H}(Q)=\int_{0}^{T}\sum_{j=0}^{n}%
\mathcal{L}_{j}(Q_{j}(t),\dot{Q}_{j}(t))dt+\int_{0}^{T}h(Q(t))dt, 
\end{equation*}%
with 
\begin{equation}
\mathcal{L}_{0}(Q_{0},\dot{Q}_{0})=K_{0}(\dot{Q}_{0})+U_{0}(Q_{0})=\frac{1}{2%
}\sum_{j=1}^{n}M_{j}\Vert \dot{Q}_{0,j}\Vert ^{2}+\sum_{\substack{ %
j,j^{\prime }=1  \\ j<j^{\prime }}}^{n}M_{j}M_{j^{\prime }}\phi _{\alpha
}\left( \Vert Q_{0,j}-Q_{0,j^{\prime }}\Vert \right) ,
\label{Lagrangian L0(Q)}
\end{equation}%
\begin{equation}
\mathcal{L}_{j}(Q_{j},\dot{Q}_{j})=K_{j}(\dot{Q}_{j})+U_{j}(Q_{j})=\frac{1}{2%
}\sum_{k=1}^{k_{j}}m_{j,k}\Vert \dot{Q}_{j,k}\Vert ^{2}+\sum_{\substack{ %
k,k^{\prime }=1  \\ k<k^{\prime }}}^{k_{j}}m_{j,k}m_{j,k^{\prime }}\phi
_{\alpha }\left( \Vert Q_{j,k}-Q_{j,k^{\prime }}\Vert \right) ,
\label{Lagrangian Lj(Q)}
\end{equation}%
for $j=1,\dots ,n$ and%
\begin{equation}
h(Q)=\sum_{\substack{ j,j^{\prime }=1  \\ j<j^{\prime }}}^{n}\sum_{\substack{
k\in K_{j}  \\ k^{\prime }\in K_{j^{\prime }}}}m_{j,k}m_{j^{\prime
},k^{\prime }}\left( \phi _{\alpha }\left( \Vert \left(
Q_{0,j}-Q_{0,j^{\prime }}\right) +\left( Q_{j,k}-Q_{j^{\prime },k^{\prime
}}\right) \Vert \right) -\phi _{\alpha }\left( \Vert Q_{0,j}-Q_{0,j^{\prime
}}\Vert \right) \right) ,  \label{h(Q)}
\end{equation}%
where $K_{j}=\left\{ 1,\dots ,k_{j}\right\} $ is the set of indices for the $%
j$-cluster.
\end{proposition}

\begin{proof}
It suffices to show that $\mathcal{L}(Q,\dot{Q})=\sum_{j=0}^{n}\mathcal{L}%
_{j}(Q_{j},\dot{Q}_{j})+h(Q)$. We first compute the kinetic energy $K$ in
these new coordinates. We get 
\begin{equation*}
K=\frac{1}{2}\sum_{j=1}^{n}\sum_{k=1}^{k_{j}}m_{j,k}\Vert \dot{q}_{j,k}\Vert
^{2}=\frac{1}{2}\sum_{j=1}^{n}\sum_{k=1}^{k_{j}}m_{j,k}\left( \Vert \dot{Q}%
_{0,j}\Vert ^{2}+\Vert \dot{Q}_{j,k}\Vert ^{2}+2\langle \dot{Q}_{0,j},\dot{Q}%
_{j,k}\rangle \right) . 
\end{equation*}%
Using the constraint \eqref{Constraint}, it follows that $%
\sum_{k=1}^{k_{j}}m_{j,k}\langle \dot{Q}_{0,j},\dot{Q}_{j,k}\rangle =0$.
Since the total mass of the $j$-cluster is $M_{j}$, we get%
\begin{equation*}
K=\frac{1}{2}\sum_{j=1}^{n}M_{j}\Vert \dot{Q}_{0,j}\Vert ^{2}+\frac{1}{2}%
\sum_{j=1}^{n}\sum_{k=1}^{k_{j}}m_{j,k}\Vert \dot{Q}_{j,k}\Vert ^{2}. 
\end{equation*}

The potential in the new coordinates becomes 
\begin{align*}
U& =\frac{1}{2}\sum_{j=1}^{n}\sum_{\substack{ k,k^{\prime }\in K_{j}  \\ %
k\neq k^{\prime }}}m_{j,k}m_{j,k^{\prime }}\phi _{\alpha }\left( \Vert
Q_{j,k}-Q_{j,k^{\prime }}\Vert \right) \\
& +\frac{1}{2}\sum_{\substack{ j,j^{\prime }=1  \\ j\neq j^{\prime }}}%
^{n}\sum _{\substack{ k\in K_{j}  \\ k^{\prime }\in K_{j^{\prime }}}}%
m_{j,k}m_{j^{\prime },k^{\prime }}\phi _{\alpha }\left( \Vert \left(
Q_{0,j}-Q_{0,j^{\prime }}\right) +\left( Q_{j,k}-Q_{j^{\prime },k^{\prime
}}\right) \Vert \right) .
\end{align*}%
The Lagrangian is of the form $\mathcal{L}(Q,\dot{Q})=\sum_{j=0}^{n}\mathcal{%
L}_{j}(Q_{j},\dot{Q}_{j})+h(Q)$, where%
\begin{eqnarray*}
h(Q) &=&\frac{1}{2}\sum_{\substack{ j,j^{\prime }=1  \\ j\neq j^{\prime }}}%
^{n}\sum_{\substack{ k\in K_{j}  \\ k^{\prime }\in K_{j^{\prime }}}}%
m_{j,k}m_{j^{\prime },k^{\prime }}\phi _{\alpha }\left( \Vert \left(
Q_{0,j}-Q_{0,j^{\prime }}\right) +\left( Q_{j,k}-Q_{j^{\prime },k^{\prime
}}\right) \Vert \right) \\
&&-\frac{1}{2}\sum_{\substack{ j,j^{\prime }=1  \\ j\neq j^{\prime }}}%
^{n}M_{j}M_{j^{\prime }}\phi _{\alpha }\left( \Vert Q_{0,j}-Q_{0,j^{\prime
}}\Vert \right) .
\end{eqnarray*}%
Replacing $M_{j}$ by $\sum_{k\in K_{j}}m_{j,k}$, we obtain \eqref{h(Q)}.
\end{proof}

\subsection{Rotating coordinates}
We
define rotating coordinates by 
\begin{align*}
Q_{j}(t)& =\exp (t\omega _{j}\mathcal{J}_{k_{j}})v_{j}(t),\quad j=1,\dots ,n
\\
Q_{0}(t)& =\exp (t\mathcal{J}_{n})v_{0}(t),
\end{align*}%
where $\omega _{j}$ is  a frequency of rotation for each cluster of bodies. The endomorphism $J$ denotes the standard complex structure on $E$ and $%
\mathcal{J}_{n}$ denotes the endomorphism $J\oplus \dots \oplus J$ of $E^{n}$%
. The components of $v_{0}=(v_{0,1},\dots ,v_{0,n})$ correspond to the $n$
positions of the centers of mass of the clusters in a rotating frame of
frequency one. The components of $v_{j}=\left( v_{j,1},\dots
,v_{j,k_{j}}\right) $ correspond to the positions of the $k_{j} $ bodies in
the $j$-cluster in a rotating frame of frequency $\omega _{j} $. By %
\eqref{Constraint} and \eqref{center of
mass} they satisfy the constraints 
\begin{equation}
\sum_{j=1}^{n}M_{j}v_{0,j}=0\quad \mbox{and}\quad%
\sum_{k=1}^{k_{j}}m_{j,k}v_{j,k}=0\text{,}  \label{constraint v}
\end{equation}%
which corresponds to $v=(v_{0},v_{1},\dots ,v_{n}) \in \E$. 

The
unperturbed functional 
\begin{equation*}
\mathcal{A}_{0}(Q)=\int_{0}^{T}\sum_{j=0}^{n}\mathcal{L}_{j}(Q_{j}(t),\dot{Q}%
_{j}(t))dt, 
\end{equation*}
is invariant under the transformations that rotate simultaneously each of
the coordinates $Q_{j,k}$. The term $\mathcal{A}_{0}(v)$ is obtained by
replacing the coordinates $Q_{j,k}$ in \eqref{Lagrangian L0(Q)} and %
\eqref{Lagrangian Lj(Q)} by the coordinates $v_{j,k}$. If we denote by $I $
the identity matrix on $E$, we define the endomorphisms $\mathcal{M}_{j}\in %
\mbox{End}(E^{k_{j}})$ and $\mathcal{M}_{0}\in \mbox{End}(E^{n})$ by 
\begin{equation*}
\mathcal{M}_{j}=m_{j,1}I\oplus \dots \oplus m_{j,k_{j}}I\quad \mbox{and}%
\quad \mathcal{M}_{0}=M_{1}I\oplus \dots \oplus M_{n}I. 
\end{equation*}%
The variations for $\mathcal{A}_{0}$ with respect to $\delta v_j$ are 
\begin{align}
\frac{\delta \mathcal{A}_{0}}{\delta v_{j}}& =\mathcal{M}_{j}\left( \mathcal{%
I}_{k_{j}}\partial _{t}+\omega _{j}\mathcal{J}_{k_{j}}\right)
^{2}v_{j}+\nabla _{v_{j}}U_{j}(v),\quad j=1,\dots ,n,  \label{EL: v tilde}
\\
\frac{\delta \mathcal{A}_{0}}{\delta v_{0}}& =\mathcal{M}_{0}\left( \mathcal{%
I}_{n}\partial _{t}+\mathcal{J}_{n}\right) ^{2}v_{0}+\nabla
_{v_{0}}U_{0}(v).  \label{EL: v}
\end{align}%
The potentials 
\begin{equation*}
U_{0}(v)=\sum_{\substack{ j,j^{\prime }=1  \\ j<j^{\prime }}}%
^{n}M_{j}M_{j^{\prime }}\phi _{\alpha }\left( \Vert v_{0,j}-v_{0,j^{\prime
}}\Vert \right) \quad \mbox{and}\quad U_{j}(v)=\sum_{\substack{ k,k^{\prime
}=1  \\ k<k^{\prime }}}^{k_{j}}m_{j,k}m_{j,k^{\prime }}\phi _{\alpha }\left(
\Vert v_{j,k}-v_{j,k^{\prime }}\Vert \right) . 
\end{equation*}
where obtained in Proposition \ref{Prop: cabling
coord}. The variations \eqref{EL:
v tilde} correspond to $k_{j}$-body problems in rotating frame,  and \eqref{EL:
v} corresponds to an $n$-body problem in rotating frame.

\begin{definition}
The \textbf{\itshape amended potential} $V_{j}:E^{k_j}\rightarrow \mathbb{R}$
for the $k_{j}$-body problem is defined by 
\begin{equation}
V_{j}(v_{j})=\frac{1}{2}\left\Vert \mathcal{M}_{j}^{1/2}v_{j}\right\Vert
^{2}+U_{j}(v_{j}).  \label{defn: amended potential}
\end{equation}%
A configuration $a_{j}\in E_{j}$ is a \textbf{\itshape central configurations%
} of the $k_{j}$-body problem (with frequency one and zero center of mass)
if it is is a critical point of $V_{j}$.
\end{definition}

A central configuration $a_{j}\in E_{j}$ of the $k_{j}$-body problem satisfies the equation $\mathcal{M}%
_{j}a_{j}=\nabla _{v_{j}}U_{j}(a_{j})$. Similarly, a central configuration $%
a_{0}\in E_{0}$ of the $%
n $-body problem satisfies $\mathcal{M}_{0}a_{0}=\nabla _{v_{0}}U_{0}(a_{0})$.
These equations are equivalent to \eqref{ccj} and \eqref{cc0}.

\subsection{Time and space scaling}

Given a central configuration $a_j=(a_{j,1}, \dots, a_{j,k_j})$ of the $k_j$%
-body problem, the scaling $v_{j}=r_{j}a_{j}$ is a central configuration if the frequency $\omega _{j}$ satisfies $%
\omega _{j}^{2}=r_{j}^{-(\alpha +1)}$. Indeed 
\begin{equation*}
\mathcal{M}_{j}\left( \mathcal{I}_{k_{j}}\partial _{t}+\omega _{j}\mathcal{J}%
_{k_{j}}\right) ^{2}v_{j}=-\omega _{j}^{2}r_{j}\mathcal{M}%
_{j}a_{j}=-\omega_j^2r_{j}^{\alpha +1}\nabla _{v_{j}}U_{j}(v_{j})\text{.} 
\end{equation*}%
For $\nu\in \mathbb{R}$, we define new coordinates $u_0=(u_{0,1}, \dots,
u_{0,n})$ and $u_j=(u_{j,1}, \dots, u_{j,k_j})$ by setting 
\begin{eqnarray*}
v_{j}(t) &=&r_{j}u_{j}(\nu t),\quad j=1,\dots,n \\
v_{0}(t) &=&u_{0}(\nu t).
\end{eqnarray*}

We shall rewrite the action functional $\mathcal{A}=\mathcal{A}_0+\mathcal{H}
$ with respect to the $u$-coordinates and find some restrictions on the set
of frequencies $\nu$ and $\omega_j$ so that the functional $\mathcal{A}$ is $%
2\pi$-periodic with respect to the new time-parameter $s=\nu t$. We first
start looking at the unperturbed functional $\mathcal{A}_0$. The $u$-coordinates are related to the $q$-coordinates as follows%
\begin{equation}
q_{j,k}(t)=\exp (tJ)u_{0,j}(\nu t)+r_{j}\exp (\omega _{j}tJ)u_{j,k}(\nu t),
\label{solutions}
\end{equation}%
for $k=1,\dots ,k_{j}$ and $j=1,\dots ,n$, with the constraint $u\in \E$.

\begin{proposition}
The unperturbed functional $\mathcal{A}_{0}(u)$ is given by 
\begin{equation}
\mathcal{A}_{0}(u)=\int_{0}^{\nu T}\left( \mathcal{L}_{0}(u_{0},\dot{u}%
_{0})+\sum_{j=1}^{n}r_{j}^{1-\alpha }\mathcal{L}_{j}(u_{j},\dot{u}%
_{j})\right) ds
\end{equation}%
where%
\begin{align}
& \mathcal{L}_{j}(u_{j},\dot{u}_{j})=\frac{1}{2}\left\Vert \mathcal{M}%
_{j}^{1/2}\left( \frac{\nu }{\omega _{j}}\partial _{s}+\mathcal{J}%
_{k_{j}}\right) u_{j}(s)\right\Vert ^{2}+U_{j}(u_{j})  \label{Lj, L0} \\
& \mathcal{L}_{0}(u_{0},\dot{u}_{0})=\frac{1}{2}\Vert \mathcal{M}%
_{0}^{1/2}\left( \nu \partial _{s}+\mathcal{J}_{n}\right) u_{0}(s)\Vert
^{2}+U_{0}(u_{0}).  \notag
\end{align}
\end{proposition}

\begin{proof}
When $\alpha >1$ the potential $\phi _{\alpha }$ is homogeneous of degree $%
1-\alpha $. Then $U_{j}(v_{j})=r_{j}^{1-\alpha }U_{j}(u_{j})$ and the
kinetic energy is%
\begin{eqnarray*}
\left\Vert \mathcal{M}_{j}^{1/2}\left( \partial _{t}+\omega _{j}\mathcal{J}%
_{k_{j}}\right) v_{j}(t)\right\Vert ^{2} &=&\left\Vert r_{j}\mathcal{M}%
_{j}^{1/2}\left( \nu \partial _{s}+\omega _{j}\mathcal{J}_{k_{j}}\right)
u_{j}(s)\right\Vert ^{2} \\
&=&r_{j}^{1-\alpha }\left\Vert \mathcal{M}_{j}^{1/2}\left( \frac{\nu }{%
\omega _{j}}\partial _{s}+\mathcal{J}_{k_{j}}\right) u_{j}(s)\right\Vert
^{2}.
\end{eqnarray*}%
For $j=0$ we have $U_{0}(v_{0})=U_{0}(u_{0})$ and the kinetic energy is 
\begin{equation*}
\left\Vert \mathcal{M}_{0}^{1/2}\left( \partial _{t}+\mathcal{J}_{n}\right)
v_{0}(t)\right\Vert ^{2}=\left\Vert \mathcal{M}_{0}^{1/2}\left( \nu \partial
_{s}+\mathcal{J}_{n}\right) u_{0}(s)\right\Vert ^{2}. 
\end{equation*}%
The case $\alpha =1$ is similar, but now for $j=1,\dots ,n$ we have $%
U_{j}(v_{j})=U_{j}(u_{j})+\ln (r_{j})$ and 
\begin{equation*}
\left\Vert \mathcal{M}_{j}^{1/2}\left( \partial _{t}+\omega _{j}\mathcal{J}%
_{k_{j}}\right) v_{j}(t)\right\Vert ^{2}=\left\Vert \mathcal{M}%
_{j}^{1/2}\left( \frac{\nu }{\omega _{j}}\partial _{s}+\mathcal{J}%
_{k_{j}}\right) u_{j}(s)\right\Vert ^{2}. 
\end{equation*}%
The result follows by rescaling $\mathcal{A}$ and adding a constant to it.
\end{proof}

\begin{proposition}
The nonlinear term $h(u(s))$ of $\mathcal{H}(u)=\int_{0}^{2\pi }h(u(s))ds$ is $2\pi $-periodic
with respect to the time variable $s$ if and only if 
\begin{equation}
\omega _{j}=1+p_{j}\nu ,\qquad p_{j}\in \mathbb{Z},  \label{om}
\end{equation}%
for each $j=1,\dots ,n$. In this case, 
\begin{eqnarray}
h(u(s)) &=&\sum_{j<j^{\prime }}\sum_{\substack{ k\in K_{j} \\ k^{\prime
}\in K_{j^{\prime }}}}m_{j,k}m_{j^{\prime },k^{\prime }}\phi _{\alpha
}\left( \Vert u_{0,j}(s)-u_{0,j^{\prime }}(s)+r_{j}\exp \left( p_{j}s%
\mathcal{J}\right) u_{j,k}(s)-r_{j^{\prime }}\exp \left( p_{j^{\prime }}s%
\mathcal{J}\right) u_{j^{\prime },k^{\prime }}(s)\Vert \right)   \notag
\label{h(u(s))} \\
&&-\sum_{j<j^{\prime }}\sum_{\substack{ k\in K_{j} \\ k^{\prime }\in
K_{j^{\prime }}}}m_{j,k}m_{j^{\prime },k^{\prime }}\phi _{\alpha }\left(
\Vert u_{0,j}(s)-u_{0,j^{\prime }}(s)\Vert \right) .
\end{eqnarray}
\end{proposition}

\begin{proof}
In \eqref{h(Q)} the terms $\phi _{\alpha }\left( \Vert \left(
Q_{0,j}(t)-Q_{0,j^{\prime }}(t)\right) +\left( Q_{j,k}(t)-Q_{j^{\prime
},k^{\prime }}(t)\right) \Vert \right) $ become 
\begin{equation*}
\phi _{\alpha }\left( \Vert u_{0,j}(s)-u_{0,j^{\prime }}(s)+r_{j}\exp \left( 
\frac{\omega _{j}-1}{\nu }s\mathcal{J}\right) u_{j,k}(s)-r_{j^{\prime }}\exp
\left( \frac{\omega _{j^{\prime }}-1}{\nu }s\mathcal{J}\right) u_{j^{\prime
},k^{\prime }}(s)\Vert \right), 
\end{equation*}%
and the terms $\phi _{\alpha }\left( \Vert Q_{0,j}(t)-Q_{0,j^{\prime
}}(t)\Vert \right) $ become $\phi _{\alpha }\left( \Vert
u_{0,j}(s)-u_{0,j^{\prime }}(s)\Vert \right) .$ It follows that the
integrand $h(u(s))$ is $2\pi $-periodic with respect to $s$ if and only
if, for all $j=1,\dots ,n$, the frequency $(\omega _{j}-1)/\nu $ is an
integer. We can thus fix $n$ integers $p_{j}$ such that $\omega
_{j}=1+p_{j}\nu $. Replacing $(\omega _{j}-1)/\nu $ by $p_{j}$ in the above
expression yields \eqref{h(u(s))}.
\end{proof}

Thus, we obtain the action functional
\begin{equation}
\mathcal{A}(u)=\mathcal{A}_{0}(u)+\mathcal{H}(u)=\int_{0}^{2\pi }\left( 
\mathcal{L}_{0}(u_{0},\dot{u}_{0})+\sum_{j=1}^{n}r_{j}^{1-\alpha }\mathcal{L}%
_{j}(u_{j},\dot{u}_{j})\right) ds+\int_{0}^{2\pi }h(u(s))ds ,
\label{Euler functional coord u}
\end{equation}%
with $h(u(s))$ as in \eqref{h(u(s))} and $\mathcal{L}_{0}(u_{0},\dot{u}%
_{0})$ and $\mathcal{L}_{j}(u_{j},\dot{u}_{j})$ as in \eqref{Lj, L0} . At
this point, the frequency $\nu \in \mathbb{R}$ is still a free parameter. In
analogy with \cite{Braids}, we choose the frequency $\nu $ as a function of $%
\varepsilon$ such that the relation $r_{j}=\varepsilon $ holds in the case
that $p_{j}=1$. Thus, we fix arbitrary integers $p_{1},\dots ,p_{n}\in 
\mathbb{Z}$ and impose the following conditions

\begin{itemize}
\item[\textbf{(A)}] $\omega _{j}=1+p_{j}\nu $ and $r_{j}=\left( 1+p_{j}\nu
\right) ^{-2/(\alpha +1)}$ for each $j=1,\dots ,n$.

\item[\textbf{(B)}] $\nu =\varepsilon ^{-\left( \alpha +1\right) /2}-1$ for
some $\varepsilon >0$.\label{nu1}
\end{itemize}

These conditions allow us to express the parameters $\omega _{j},r_{j},\nu $
as functions of $\varepsilon $ with
\begin{eqnarray}
r_{j} &=&\left( 1+p_{j}\nu \right) ^{-2/(\alpha +1)}=p_{j}^{-2/(\alpha
+1)}\varepsilon +\mathcal{O}\left( \varepsilon ^{\left( \alpha +3\right)
/2}\right) =p_{j}^{-2/(\alpha +1)}\varepsilon +\mathcal{O}\mathbb{(}%
\varepsilon ^{2}),  \label{expansion} \\
\omega _{j}/\nu  &=&p_{j}+1/\nu =p_{j}+\mathcal{O(}\varepsilon ^{\left(
\alpha +1\right) /2})=p_{j}+\mathcal{O}\mathbb{(}\varepsilon ).  \notag
\end{eqnarray}%
We look at solutions of the $N$-body problem as critical points of $\mathcal{A}(u)$ on some
collision-less open set $\Omega $ for small $\varepsilon $. 
\subsection{Gradient formulation and symmetries}

The space $H^{1}(S^{1},E^{n+N})$ is identified with its dual $%
H^{1}(S^{1},E^{n+N})^{\ast }$ by the Riesz representation Theorem. This
allows us to define the gradient operator by the relation $\nabla \mathcal{A}%
=(-\partial _{s}^{2}+1)^{-1}\delta \mathcal{A}$. We use the gradient
formulation $\nabla \mathcal{A}=\nabla \mathcal{A}_{0}+\nabla \mathcal{H}$,
where $\nabla \mathcal{H}=\mathcal{O}(\varepsilon )$ is a
compact operator and $\nabla \mathcal{A}_{0}$ is given by
\begin{align}
\nabla _{u_{j}}\mathcal{A}_{0}(u)& =(-\partial
_{s}^{2}+1)^{-1}r_{j}^{1-\alpha }\left( -\mathcal{M}_{j}\left( \frac{\nu }{%
\omega _{j}}\mathcal{I}_{k_{j}}\partial _{s}+\mathcal{J}_{k_{j}}\right)
^{2}u_{j}+\nabla _{u_{j}}U_{j}(u_{j})\right) ,  \label{grad uj} \\
\nabla _{u_{0}}\mathcal{A}_{0}(u)& =(-\partial _{s}^{2}+1)^{-1}\left( -%
\mathcal{M}_{0}\left( \nu \mathcal{I}_{n}\partial _{s}+\mathcal{J}%
_{n}\right) ^{2}u_{0}+\nabla _{u_{0}}U_{0}(u_{0})\right) .  \label{grad u0}
\end{align}%
We know that these equations admit the constant solution $%
u_{a}=(a_{0},a_{1},\dots ,a_{n})\in X$ where $a_{0}\in E_{0}$ is a central
configuration of the $n$-body problem, each $a_{j}\in E_{j}$ is a central
configuration of the $k_{j}$-body problem for $j=1,...,n_{0}$ and $a_{j}\in
E_{j}=\{0\}$ for $j=n_{0}+1,...,n$.

Due to the presence of symmetries, there is in fact a group orbit of
solutions generated by $u_{a}$. The functional $\mathcal{A}_{0}(u)$ is
invariant under the action of the torus $G=U(1)^{n+1}$  defined by 
\begin{equation*}
(g_{0},g_{1},\dots ,g_{n})\cdot (u_{0},u_{1}\dots
,u_{n})=(g_{0}u_{0},g_{1}u_{1},\dots ,g_{n}u_{n}).
\end{equation*}%
where $(g_{0},g_{1},\dots ,g_{n})\in U(1)^{n+1}$. The action of $g_{0}$
rotates the $n$-body problem consisting of the $n$ centers of mass of the
clusters about the origin. The action of each $g_{j}$ rotates the $k_{j}$%
-bodies in each cluster about their center of mass for $j=1,...,n_{0}$ and
acts trivially otherwise. The coupling term $\mathcal{H}$ breaks this
symmetry in the sense that the perturbed functional $\mathcal{A}$ is
invariant under the diagonal subgroup $H=\widetilde{U(1)}$. By $G$%
-equivariance of the equations \eqref{grad uj} and \eqref{grad u0}, the
group orbit $G(u_{a})$ is an orbit of solutions. 
\subsection{Euler-Lagrange equations with holonomic constraints}

Let $S^{1}=\mathbb{R}/2\pi \mathbb{Z}$ be the standard parametrization of
the circle and denote by 
\begin{equation*}
X=H^{1}(S^{1},\E)\subset H^{1}(S^{1},E^{n+N})
\end{equation*}%
the real Hilbert space of $2\pi $-periodic paths in $\E$. 
The solutions of the $N$-body problem are the critical points of the
augmented action $\mathcal{A}$ restricted to $X$. That is, the system of equations of the $N$-body problems is the gradient of $\mathcal{A}$ taking respect the subspace $X$ and is given by 
\begin{equation}
P_{X}\nabla \mathcal{A}(u)=0.  \label{proj}
\end{equation}%
Solving this system is equivalent to finding the critical points of $\mathcal{A}$
with the
holonomic constraints 
\begin{equation*}
g_{0}^{s}(u_{0})=\sum_{j=1}^{n}M_{j}u_{0,j}\cdot e_{s}=0,\qquad
g_{j}^{s}(u_{j})=\sum_{k=1}^{k_{j}}m_{j,k}u_{j,k}\cdot e_{s}=0,\qquad s=1,2.
\end{equation*}%
Thus, the explicit projection $P_{X}:H^{1}(S^{1},E^{n+N})\rightarrow X$ is given in
components $u_{j}$ by 
\begin{equation}\label{px}
P_{X}(u_{j})=u_{j}-\sum_{s=1}^{2}\frac{(u_{j},\nabla
_{u_{j}}g_{j}^{s}(u_{j}))}{\Vert \nabla _{u_{j}}g_{j}^{s}(u_{j})\Vert ^{2}}%
\nabla _{u_{j}}g_{j}^{s}(u_{j}),\qquad j=0,\dots ,n,
\end{equation}
and the explicit system of equations, equivalent to \eqref{proj}, is
\begin{equation*}
\nabla _{u_{j}}\mathcal{A}(u)=\sum_{s=1}^{2}\frac{(\nabla _{u_{j}}\mathcal{A}%
(u),\nabla _{u_{j}}g_{j}^{s}(u_{j}))}{\Vert \nabla
_{u_{j}}g_{j}^{s}(u_{j})\Vert ^{2}}\nabla _{u_{j}}g_{j}^{s}(u_{j}),\qquad
j=0,\dots ,n.
\end{equation*}%
with the left hand side given in \eqref{grad u0} and \eqref{grad uj}. 
We study these equations  in a collision-less tubular neighbourhood $\Omega \subset X$ of the orbit $%
G(u_{a})$.

\begin{remark}
One may consider also the augmented action with the
holonomic constraints 
\begin{equation*}
\mathcal{A}^{\ast }(u^{\ast })=\mathcal{A}_{0}(u)+\mathcal{H}%
(u)+\sum_{j=0}^{n}\sum_{s=1}^{2}\lambda _{j}^{s}\cdot g_{j}^{s}\text{,}
\end{equation*}%
where $u^{\ast }=(u,\lambda _{0},...,\lambda _{n})\in E^{N+n}\times \mathbb{R%
}^{2(n+1)}$. Solving the action for the augmented system $\nabla\mathcal{A}^{\ast}(u^{\ast })=0$ is equivalent to
solving $\nabla_{u_j}\mathcal{A}^{\ast}(u^{\ast })=P_{X}%
\nabla_{u_j}\mathcal{A}(u)=0$ with $u\in E_{red}^{n+N}$
(because $\nabla_{\lambda_j^s}\mathcal{A}^{\ast}(u^{\ast })=g_{j}^{s}(u)=0$). For each solution $u_{a}$ such that $\nabla_u\mathcal{A}_0(u_{a})=0$ there is a unique $u_{a}^{\ast }$ such
that $\nabla_{u^{\ast}}\mathcal{A}_0^{\ast}(u_{a}^{\ast
})=0$, and a similar procedure can be implemented for the orbit of $%
u_{a}^{\ast }$ with the Lagrange multipliers $\lambda_j $ given as variables. The
procedures are equivalent because $(u,\lambda )$ is in the kernel of $\nabla _{u^{\ast }}^{2}\mathcal{A}^{\ast }(u_{a}^{\ast })$ if and only $%
u$ is in the kernel of $P_{X}\nabla _{u}^{2}\mathcal{A}(u_{a})|_{X}$.
\end{remark}

\begin{remark}
\label{remark:ambiant} We may consider also coordinates for $\E$ to write directly the action in these coordinates. That is, we may fix coordinates $w_{j}=(w_{j,1},\dots
,w_{j,k_{j}-1})$ on $E_{j}$ such that $u_{j}\in E_{j}$ viewed as an element
of $E^{k_{j}}$ can be written of the form $u_{j}=\Lambda _{j}w_{j}$ for some 
$k_{j}\times (k_{j}-1)$ matrix $\Lambda _{j}$. Note that 
\begin{equation*}
\left\langle \nabla _{u_{j}}\mathcal{A}_{0}(u_{j}),\delta u_{j}\right\rangle
=\left\langle \Lambda _{j}^{T}\nabla _{u_{j}}\mathcal{A}_{0}(\Lambda
_{j}w_{j}),\delta w_{j}\right\rangle 
\end{equation*}
where the first inner product is taken on the ambient space $H^{1}([0,2\pi
],E^{k_{j}})$ and the second on the reduced space $H^{1}([0,2\pi ],E_{j})$. The reduced
Euler-Lagrange equations on $X$ are then 
\begin{equation}
\Lambda _{j}^{T}\nabla _{u_{j}}\mathcal{A}_{0}(\Lambda _{j}w_{j})=0,\qquad
j=0,\dots ,n-1.  \label{EL eqns reduced}
\end{equation}

This is the method adopted in our previous paper \cite{Braids} for $j=1$ and $k_{j}=2$. In that case, the Euler-Lagrange equations of the 2-body problem
in rotating frame are 
\begin{equation}
-\mathcal{M}_{1}\left( \frac{\nu }{\omega _{1}}\mathcal{I}_{k_{1}}\partial
_{s}+\mathcal{J}_{k_{1}}\right) ^{2}u_{1}+\nabla _{u_{1}}U_{1}(u_{1})=0
\label{2BP}
\end{equation}%
where $\mathcal{M}_{1}$ is the diagonal matrix whose entries are the masses
of the two bodies $m_{1,1}$ and $m_{1,2}$. The positions are denoted $%
u_{1}=(u_{1,1},u_{1,2})$. We parametrize the reduced space $E_{1}$ by the
relative position $w_{1}=u_{1,1}-u_{1,2}$, so that $u_{1}=\Lambda _{1}w_{1}$
with $\Lambda _{1}=%
\begin{bmatrix}
\lambda _{1,1}I & \lambda _{1,2}I%
\end{bmatrix}%
$, with 
\begin{equation*}
\lambda _{1,1}=\frac{m_{1,2}}{m_{1,1}+m_{1,2}},\qquad \lambda _{2,1}=-\frac{%
m_{1,1}}{m_{1,1}+m_{1,2}}.
\end{equation*}%
Conjugating equation \eqref{2BP} by $\Lambda _{1}^{T}$ on the left and $%
\Lambda _{1}$ on the right yields the Kepler problem in rotating frame
\begin{equation*}
-M_{0}\left( \frac{\nu }{\omega _{1}}\partial _{s}+J\right)
^{2}w_{1}-m_{1,1}m_{1,2}\frac{w_{1}}{\Vert w_{1}\Vert ^{\alpha +1}}=0,
\end{equation*}%
where $M_{0}=\frac{m_{1,1}m_{1,2}}{%
m_{1,1}+m_{1,2}}$ is the reduced mass.
\end{remark}

\section{Lyapunov-Schmidt reduction}

In this section we reduce the problem to finite dimension by writing the paths in Fourier series and applying a Lyapunov-Schmidt reduction. We have that
\begin{equation*}
X=\left\{ u\in L^{2}(S^{1},\E)\mid \sum_{\ell \in \mathbb{Z}}(\ell
^{2}+1)\Vert \hat{u}_{\ell }\Vert ^{2}<\infty \right\},
\end{equation*}%
where $(\hat{u}_{\ell })$ is the sequence of Fourier coefficients in $(\E)^{%
\mathbb{C}}=\E\oplus i\E$ satisfying $\hat{u}_{\ell }=\overline{\hat{u}}%
_{-\ell }$. That is, $u\in X$ has Fourier series $u=\sum_{\ell \in \mathbb{Z}%
}\hat{u}_{\ell }e_{\ell }$ where $e_{\ell }:S^{1}\rightarrow \mathbb{C}$ is
given by $e_{\ell }(s)=e^{i\ell s}$. We can then write $X=X_{0}\oplus W$ and
any element $u\in X$ decomposes uniquely as $u=\xi +\eta $ with 
\begin{equation*}
\xi =\hat{u}_{0},\qquad \eta =\sum_{\ell \neq 0}\hat{u}_{\ell }e_{\ell }.
\end{equation*}%

The system of equations $P_{X}\nabla \mathcal{A}(\xi +\eta )=0$ splits into 
\begin{align}
P_{X_{0}}\nabla \mathcal{A}(\xi +\eta )& =0\in X_{0}~,  \nonumber \\
P_{W}\nabla \mathcal{A}(\xi +\eta )& =0\in W,  \label{LS second eq}
\end{align}%
where $P_{X_{0}}:H^1(S^1, E^{n+N})\to X_0$ is the canonical projection from $P:X\to X_{0}$ given by $P u=\xi$, composed with $P_X$ in \eqref{px}.
The projection $P_{W}: H^1(S^1, E^{n+N})\to W$ is defined as the canonical projection from $(I-P):X\to W$ given by $(I-P) u=\eta$, composed with $P_X$.

The Lyapunov-Schmidt
reduction requires solving the equation $P_W\nabla \mathcal{A}%
(\xi +\eta )=0$. For this purpose, we define an operator $F_{\varepsilon
}:\Omega \subset X\rightarrow W$ by 
\begin{equation*}
F_{\varepsilon }(\xi ,\eta ):=\mathcal{D}_{\varepsilon }P_W\nabla \mathcal{A}(\xi + \eta) 
\end{equation*}%
where $\mathcal{D}_{\varepsilon }\in \mbox{End}(E^{n+N})$ is the block
diagonal matrix

\begin{equation}
\mathcal{D}_{\varepsilon }=\nu ^{-2}\mathcal{I}_{n}\oplus r_{1}^{\alpha -1}%
\mathcal{I}_{k_{1}}\oplus \dots \oplus r_{n}^{\alpha -1}\mathcal{I}_{k_{n}},
\end{equation}%
with $r_{j}^{\alpha -1}=\mathcal{O}(\varepsilon ^{\alpha -1})$ and $\nu
^{-2}=\mathcal{O}(\varepsilon ^{\alpha +1})$. Since $\mathcal{D}_{\varepsilon }$ is block diagonal, it commutes with $P_W$ and we get
$$F_{\varepsilon }(\xi ,\eta )=P_W \mathcal{D}_{\varepsilon}\nabla \mathcal{A}(\xi+\eta).$$

Solving \eqref{LS second
eq} is equivalent to solving $F_{\varepsilon }(\xi ,\eta )=0$ for $%
\varepsilon \neq 0$ because $\mathcal{D}_{\varepsilon }$ is an isomorphism.
The operator $F_{\varepsilon }(\xi ,\eta )$ is continuous at 
$\varepsilon =0$ because $\lim_{\varepsilon \rightarrow 0}\left( \nu /\omega
_{j}\right) ^{2}=\left( 1/p_{j}\right) ^{2}$. The limit
\begin{equation}\label{f0}
F_{0}(\xi ,\eta )=\lim_{\varepsilon \rightarrow 0} P_W\mathcal{D}_{\varepsilon }\nabla %
\mathcal{A}_0(\xi +\eta )
\end{equation}%
is well defined since $\mathcal{D}_{\varepsilon}\nabla\mathcal{H}=\mathcal{O}(\varepsilon)$. Furthermore, $F_{0}(gu_{a},0)=0$ for all $g\in G$ by
equivariance of the unperturbed gradient. Solving $F_{\varepsilon }(\xi
,\eta )=0$ requires the derivative $\partial _{\eta }F_{0}[(gu_{a},0)]$ to
be invertible on $W$. Although this is true
when $\alpha \neq 2$, the operator is not invertible on the whole space $W$
when $\alpha =2$. We shall then treat these cases separately.

\begin{remark}
Alternatively, we could have used the matrix 
\begin{equation*}
\mathcal{D}_{\varepsilon }^{\prime }=\varepsilon ^{\alpha +1}\mathcal{I}%
_{n}\oplus \varepsilon ^{\alpha -1}\mathcal{I}_{k_{1}}\oplus \dots \oplus
\varepsilon ^{\alpha -1}\mathcal{I}_{k_{n}} 
\end{equation*}
according to \cite{Braids}. Both regularizations allow us to perform the
same reduction, the only difference is that the scaling matrix $\mathcal{D}%
_{\varepsilon }$ depends on $p_{j}$'s.
\end{remark}

\subsection*{The case $\protect\alpha \neq 2$}

The linearization of \eqref{f0} at $u_{a}$ is given by

$$\partial _{\eta }F_{0}[(u_{a},0)]=\lim_{\varepsilon \rightarrow 0} P_W\mathcal{D}_{\varepsilon }\nabla_u^2 %
\left.\mathcal{A}_0[u_a]\right|_{W} .$$
The Hessian operator of $\mathcal{A}_0$ at the critical point is block diagonal of the form
\begin{equation*}
\nabla _{u}^{2}\mathcal{A}_{0}[u_{a}]=\nabla _{u_{0}}^{2}\mathcal{A}%
_{0}[a_{0}]\oplus \nabla _{u_{1}}^{2}\mathcal{A}_{0}[a_{1}]\oplus \dots
\oplus \nabla _{u_{n}}^{2}\mathcal{A}_{0}[a_{n}].
\end{equation*}%
The blocks are derived using \eqref{grad uj}, \eqref{grad u0} and given by
\begin{align*}
\nabla _{u_{j}}^{2}\mathcal{A}_{0}[a_{j}]& =(-\partial
_{s}^{2}+1)^{-1}r_{j}^{\alpha -1}\left( -\left( \nu /\omega _{j}\right) ^{2}%
\mathcal{M}_{j}\partial _{s}^{2}-2\left( \nu /\omega _{j}\right) \mathcal{M}%
_{j}\mathcal{J}_{k_{j}}\partial _{s}+\nabla _{u_{j}}^{2}V_{j}[a_{j}]\right) ,
\\
\nabla _{u_{0}}^{2}\mathcal{A}_{0}[a_{0}]& =(-\partial _{s}^{2}+1)^{-1}\nu
^{2}\left( -\mathcal{M}_{0}\partial _{s}^{2}-2\nu ^{-1}\mathcal{M}_{0}%
\mathcal{J}_{n}\partial _{s}+\nu ^{-2}\nabla _{u_{0}}^{2}V_{0}[a_{0}]\right)
,
\end{align*}%
where $V_{j}$ are the amended potentials in 
\eqref{defn: amended
potential}.

\begin{definition}
We define the \defn{regularized action} for the $j$-cluster by $%
A_{j}(u_{j})=\int_{0}^{2\pi }L_{j}(u_{j},\dot{u}_{j})ds$ where
\begin{equation}
L_{j}(u_{j},\dot{u}_{j})=\lim_{\varepsilon \rightarrow 0}\mathcal{L}%
_{j}(u_{j},\dot{u}_{j})=\frac{1}{2}\left\Vert \mathcal{M}_{j}^{1/2}\left( 
\frac{1}{p_{j}}\partial _{s}+\mathcal{J}_{k_{j}}\right) u_{j}(s)\right\Vert
^{2}+U_{j}(u_{j}).  \label{Lj}
\end{equation}
\end{definition}

Since $\lim_{\varepsilon \rightarrow 0}\left( \omega _{j}/\nu \right)
=1/p_{j}$, we obtain 
\begin{equation*}
\lim_{\varepsilon \rightarrow 0}r_{j}^{\alpha -1}\nabla _{u_{j}}\mathcal{A}%
_{0}(u_{j})=\nabla _{u_{j}}A_{j}(u_{j}),
\end{equation*}%
and 
\begin{equation*}
\partial _{\eta }F_{0}[(u_{a},0)]=P_{W}\left( \left.-(-\partial _{s}^{2}+1)^{-1}%
\mathcal{M}_{0}\partial _{s}^{2}\oplus \nabla _{u_{1}}^{2}A_{1}[a_{1}]\oplus
\dots \oplus \nabla _{u_{n}}^{2}A_{n}[a_{n}]\right|_{W}\right),
\end{equation*}%
where  
\begin{equation*}
\nabla _{u_{j}}^{2}A_{j}[a_{j}]=(-\partial _{s}^{2}+1)^{-1}\left(
-(1/p_{j})^{2}\mathcal{M}_{j}\partial _{s}^{2}-2\left( 1/p_{j}\right) 
\mathcal{M}_{j}\mathcal{J}_{k_{j}}\partial _{s}+\nabla
_{u_{j}}^{2}V_{j}[a_{j}]\right) .
\end{equation*}
Since $\eta =\sum_{\ell \neq 0}\hat{u}_{\ell }e_{\ell }\in W$, we can write 
\begin{equation}
\partial _{\eta }F_{0}[(u_{a},0)]\eta =\sum_{\ell \neq 0}\hat{T}_{\ell }\hat{u%
}_{\ell }e_{\ell },  \label{unreduced F}
\end{equation}%
where the matrix $\hat{T}_{\ell }\in \mbox{End}(\E)$ is block diagonal
of the form 
\begin{equation}
\hat{T}_{\ell }=\hat{T}_{\ell ,u_{0}}\oplus \hat{T}_{\ell ,u_{1}}\oplus
\dots \oplus \hat{T}_{\ell ,u_{n}}.  \label{G_0 blocks}
\end{equation}%
These blocks are given explicitly by%
\begin{align}
\hat{T}_{\ell ,u_{0}}& =\frac{\ell ^{2}}{\ell ^{2}+1}P_{E_0}\left.\mathcal{M}_{0}\right|_{E_0} \notag
\label{T block} \\
\hat{T}_{\ell ,u_{j}}& =\frac{1}{1+\ell ^{2}}P_{E_j}\left.\left( \left( \frac{\ell }{p_{j}%
}\right) ^{2}\mathcal{M}_{j}-2i\left( \frac{\ell }{p_{j}}\right) \mathcal{M}%
_{j}\mathcal{J}_{k_{j}}+\nabla _{u_{j}}^{2}V_{j}[a_{j}]\right)\right|_{E_j} 
\end{align}%
where $\hat{T}_{\ell ,u_{0}}\in \mbox{End}(E_0^{\mathbb{C}})$ and $\hat{T%
}_{\ell ,u_{j}}\in \mbox{End}(E_j^{\mathbb{C}})$.

\begin{definition}
\label{Def1} The central configuration $a_{j}\in E_{j}$ is 
\defn{ $2\pi
p_j$-nondegenerate} if the group orbit $U(1)(a_{j})$ is a nondegenerate
critical manifold of the functional $A_{j}(u_{j})$ defined by \eqref{Lj} in
the space $H^{1}(S^{1},E_{j})$.
\end{definition}

The orbit $U(1)(a_{j})$ is called a nondegenerate critical manifold of the
functional $A_{j}(u_{j})$ if the kernel of the Hessian at $a_j$ in $H^{1}(S^{1},E_{j})$ is $\mbox{span}(\mathcal{J}%
_{k_{j}}a_{j})$. Since $u_{j}\in H^{1}(S^{1},E_{j})$ is orthogonal to $%
\mbox{span}(\mathcal{J}_{k_{j}}a_{j})$ if and only if $u=\sum_{\ell \in 
\mathbb{Z}}\hat{u}_{\ell }e_{\ell }$ with $\hat{u}_{0}$ orthogonal to $%
\mbox{span}(\mathcal{J}_{k_{j}}a_{j})$ in $E_{j}$, this condition is
equivalent to the assumption that the blocks $\hat{T}_{\ell
,u_{j}}$ are invertible in $E_{j}^{\mathbb{C}}$ for $\ell \neq 0$
and $\hat{T}_{0,u_{j}}$ is invertible in the complement to 
$\mbox{span}(\mathcal{J}_{k_{j}}a_{j})$ in $E_{j}$.

\begin{lemma}
\label{lem: bounded inverse} Assume that $\alpha \neq 2$ and $a_{j}$ is $%
2\pi p_{j}$-nondegenerate for $j=1,\dots ,n_{0}$. Then the operator $%
\partial _{\eta }F_{0}[(gu_{a},0)]$ is invertible on $W$ for all $g\in G$,
i.e. there is a constant $c>0$ such that 
\begin{equation*}
\Vert \partial _{\eta }F_{0}[(gu_{a},0)]^{-1}\eta \Vert \leq c\Vert \eta
\Vert \quad \mbox{for each}\quad \eta \in W,~g\in G.
\end{equation*}
\end{lemma}

\begin{proof}
For $\ell \neq 0$, the block $\hat{T}_{\ell }$ in \eqref{G_0 blocks}
is always invertible and so are the blocks $\hat{T%
}_{\ell ,u_{j}}$ when $\ell \neq 0$ by assumption that $a_{j}$ is $%
2\pi p_{j}$-nondegenerate for $j=1,\dots ,n_{0}$. This implies that the
operator $\partial _{\eta }F_{0}[(u_{a},0)]$ is invertible on $W$ with
\begin{equation*}
\partial _{\eta }F_{0}[(u_{a},0)]^{-1}\eta =\sum_{\ell \neq 0}
\hat{T}_{\ell }^{-1}\hat{u}_{\ell }e_{\ell },\quad \eta \in W.
\end{equation*}%
Since $\hat{T}_{\ell ,u_{j}}\rightarrow \left( 1/p_{j}\right) ^{2}P_{E_j}\left.\mathcal{M}%
_{j}\right|_{E_j} $ for $j=1,\dots ,n$ and $\hat{T}_{\ell ,u_{0}}\rightarrow P_{E_0}\left.\mathcal{M}%
_{0}\right|_{E_0} $ when $\ell \rightarrow \infty $, it follows that 
\begin{equation}
\Vert \partial _{\eta }F_{0}[(u_{a},0)]^{-1}\eta \Vert \leq c\Vert \eta
\Vert .  \label{est}
\end{equation}%
where $c>0$ is a constant such that any eigenvalue $\lambda $ of $\hat{T}%
_{\ell }$ in \eqref{G_0 blocks} satisfies $\left\vert \lambda \right\vert
\geq c^{-1}$. Note that the Hessian $\nabla ^{2}\mathcal{A}_{0}[gu_{a}]$ is
conjugated to $\nabla ^{2}\mathcal{A}_{0}[u_{a}]$ as $\nabla \mathcal{A}%
_{0}$ is $G$-equivariant. This also holds for the constraint gradient because the $G$-action preserves the constraints. Hence $\partial _{\eta }F_{0}[(gu_{a},0)]$ and $%
\partial _{\eta }F_{0}[(u_{a},0)]$ are conjugated and the estimate %
\eqref{est} holds when replacing $u_{a}$ by $gu_{a}$ because $G$ acts by
isometries.
\end{proof}

\begin{remark}
\label{Rem} To explain further the meaning of the $2\pi p_{j}$-nondegeneracy
condition, we can consider the Hamiltonian system with Hamiltonian 
\begin{equation*}
H_{j}(u_{j},\pi _{j})=K_{j}-U_{j}
\end{equation*}%
for the $k_{j}$-body problem, where $\pi _{j}=\partial _{\dot{u}_{j}}L$ is
obtained from the Lagrangian $L_{j}=K_{j}+U_{j}$ defined in the space $E_{j}$
by means of the Legendre transformation. The relative equilibrium $a_{j}$ is
linearly stable if the eigenvalues of the linearized Hamiltonian vector
field $\mathfrak{J}^{-1}\nabla ^{2}H_{j}[(a_{j},0)]$ are all purely imaginary,
except by the double zero-eigenvalue corresponding to the generator of the
group orbit $\left( \mathcal{J}_{j}a_{j},0\right) $. On the other hands, our 
$2\pi p_{j}$-nondegenerate condition for $a_{j}$ can be verified similarly
according to the equivalent condition that the matrix $\mathfrak{J}%
^{-1}\nabla ^{2}H_{j}[(a_{j},0)]$ has no eigenvalues of the form $2\pi i\ell $
with $\ell \in \mathbb{Z}$, except by a double zero-eigenvalue corresponding
to the generator of the $U(1)$-orbit of $(a_{j},0)$.
\end{remark}

Unfortunately, the $2\pi p_j$-nondegenerate condition has not been verified
before in the literature for central configurations. In order to complement
our result, we verify this condition for an infinite number of polygonal
configurations in Section 4. We conjecture that, for $\alpha \neq 2$, the
condition of being $2\pi $-nondegenerate holds for a generic set of central
configurations in a set of parameters of masses.

\subsection*{Gravitational case $\protect\alpha=2$}

When $\alpha =2$ all the central configurations $a_{j}$ are $2\pi p_{j}$%
-nondegenerate due to the existence of elliptic homographic solutions, and
the matrices $\hat{T}_{\ell ,u_{j}}$ are never invertible for $\ell =\pm
p_{j}$. To study the case $\alpha =2$, we distinguish different type of
symmetric configurations under $2\pi /m$-rotations at the origin. Examples
of symmetric configurations that we can braid are the Maxwell configuration
and nested polygonal configurations. In these cases, we can only divide the
central body. Thus, we require the additional assumptions listed below.

\begin{itemize}
\item[\textbf{(C0)}] We consider the $N$-body problem with $N=n+k_{1}-1$,
i.e.%
\begin{equation*}
k_{1}>1,\qquad k_{2}=...=k_{n}=1. 
\end{equation*}%
Then $E_{j}=\{0\}$ for $j=2,\dots ,n$ and $\E=E_{0}\times
E_{1}$. A path $u\in X$ is then written as $u=(u_{0},u_{1})$. Denote by $S_{n}$ the permutation group of $n$ letters. We need a
group of symmetries $\Gamma $ that allows dealing with the resonances. We
consider $\Gamma <\mathbb{Z}_{m}\times S_{n}$ to be the discrete subgroup
generated by the element $(\theta ,\sigma )$ such that 
\begin{equation*}
\theta =2\pi /m\in \mathbb{Z}_{m},\qquad \sigma ^{m}=(1)\in S_{n},\qquad
\sigma (1)=1, 
\end{equation*}%
and that acts on the components of $u=(u_{0},u_{1})$ as follows: 
\begin{align*}
(\theta ,\sigma )u_{1}(s)& =u_{1}(s+\theta ), \\
(\theta ,\sigma )u_{0}(s)& =(\exp (-\theta J)u_{0,\sigma (1)}(s+\theta
),\dots ,\exp (-\theta J)u_{0,\sigma (n)}(s+\theta )).
\end{align*}

\item[\textbf{(C1)}] The next assumption is that the masses for the bodies $%
q_{j,1}$ for $j=2,\dots ,n$ satisfy 
\begin{equation*}
m_{j,1}=m_{\sigma (j),1}~,\qquad j=2,\dots ,n. 
\end{equation*}%
For the sake of simplicity, we also assume that $p_{1}=1$, which along with
conditions {\normalfont\textbf{(A)-(B)}} imply that 
\begin{equation*}
r_{1}=\varepsilon ,\qquad \omega _{1}=\varepsilon ^{-\left( \alpha +1\right)
/2},\qquad \nu =\varepsilon ^{-\left( \alpha +1\right) /2}-1. 
\end{equation*}
\end{itemize}

\begin{lemma}
\label{lem: alpha=2 invariance} For $\alpha =2$, under the assumptions $%
\normalfont{\textbf{(C0)-(C1)}}$, the functional $\mathcal{A}$ is $\Gamma $%
-invariant.
\end{lemma}

\begin{proof}
Since the variables $u_{j}(s)$ for $j=0,\dots ,n$ are uncoupled in $\mathcal{%
A}_{0}$, it is an immediate consequence of the assumptions that the
functional $\mathcal{A}_{0}$ is $\Gamma $-invariant. It remains to show that
the coupling term $\mathcal{H}$ is $\Gamma $-invariant. By assumptions 
\textbf{(C0)-(C1)}, the integrand $h$ is obtained from \eqref{h(u(s))} after
setting $u_{j,k}=0$ whenever $j\geq 2$, and $p_{1}=1$. We get 
\begin{equation*}
h(u(s))=\sum_{j^{\prime }=2}^{n}\sum_{k\in K_{1}}m_{1,k}m_{j^{\prime
},1}\left( \phi _{\alpha }\left( \Vert u_{0,1}(s)-u_{0,j^{\prime
}}(s)+r_{1}\exp \left( s\mathcal{J}\right) u_{1,k}(s)\Vert \right) -\phi
_{\alpha }\left( \Vert u_{0,1}(s)-u_{0,j^{\prime }}(s)\Vert \right) \right)
. 
\end{equation*}%
Set $s^{\prime }=s+\theta $. Since $\sigma (1)=1$ and the norms are
invariant by rotations, we obtain 
\begin{align*}
h((\theta ,\sigma )u(s))& =\sum_{j^{\prime }=2}^{n}\sum_{k\in
K_{1}}m_{1,k}m_{\sigma (j^{\prime }),1}\left( \phi _{\alpha }\left( \Vert
u_{0,1}(s^{\prime })-u_{0,\sigma (j^{\prime })}(s^{\prime })+r_{1}\exp
\left( s^{\prime }\mathcal{J}\right) u_{1,k}(s^{\prime })\Vert \right)
\right. \\
& \left. -\phi _{\alpha }\left( \Vert u_{0,1}(s^{\prime })-u_{0,\sigma
(j^{\prime })}(s^{\prime })\Vert \right) \right) =h(u(s^{\prime }))
\end{align*}%
by re-indexing the sum at the end. Finally, 
\begin{equation*}
\mathcal{H(}(\theta ,\sigma )u)=\int_{0}^{2\pi }h((\theta ,\sigma
)u(s))ds=\int_{\theta }^{2\pi +\theta }h(u(s^{\prime }))ds^{\prime }=%
\mathcal{H}(u). 
\end{equation*}
\end{proof}

Thus, by the Palais Principle of Symmetric Criticality \cite{Palais}, we can
restrict the study of critical points to the fixed point space $X^{\Gamma }$%
. Notice that a path $u=(u_{0},u_{1})$ belongs to $X^{\Gamma }$ if and only
if $u(s)=(\theta ,\sigma )u(s)$. Thus 
\begin{equation*}
X^{\Gamma }=H^{1}(S^{1},E_{0})^{\Gamma }\oplus H^{1}(S^{1},E_{1})^{\Gamma }, 
\end{equation*}
where $H^{1}(S^{1},E_{1})^{\Gamma }$ is the Sobolev space of $2\pi /m$%
-periodic functions in $E_{1}$ and $H^{1}(S^{1},E_{0})^{\Gamma }$ is the
subspace of functions $u_{0}$ satisfying the symmetry $u_{0,\sigma (j)}=\exp
(-\theta \mathcal{J})u_{0,\sigma (j)}(s+\theta )$.

\begin{itemize}
\item[\textbf{(C2)}] The last assumption (to ensure that $u_{a}\in X^{\Gamma
}$) is that the central configurations $a_{0}\in E_{0}$ satisfies the
property 
\begin{equation}
a_{0,j}=\exp (-\theta \mathcal{J})a_{0,\sigma (j)}.
\end{equation}%
Since $\sigma ^{m}=1$ and $\theta =2\pi /m$, this condition implies that the
central configuration $a_{0}$ is symmetric by $2\pi /m$-rotations in the
plane, and since $\sigma (1)=1$ that 
\begin{equation*}
a_{0,1}=\exp (-\theta \mathcal{J})a_{0,1}=0. 
\end{equation*}
\end{itemize}

This condition holds true in many symmetric configurations: Maxwell
configuration, nested polygons with a center \cite{Mont15} and spiderwebs
with a center \cite{Webs}.

\begin{definition}
\label{Def2}The central configuration $a_{1}\in E_{1}$ is 
\defn{ $2\pi
/m$-nondegenerate} if the orbit $U(1)(a_{1})$ is a nondegenerate critical
manifold of the functional $A_{1}(u_{1})$ in the fixed point space $%
H^{1}(S^{1},E_{1})^{\Gamma }$.
\end{definition}

Thus $a_{1}$ is a $2\pi /m$-nondegenerate if the kernel of the Hessian of $A_{1}$ at $a_{1}$ in $H^{1}(S^{1},E_{1})^{\Gamma }$ {is }$\mbox{span}%
(\mathcal{J}_{k_{1}}a_{1})$. This weaker condition is equivalent to the
hypothesis that the matrices $\hat{T}_{\ell ,u_{1}}$ are
invertible in $E_{1}^{\mathbb{C}}$ for $\ell \in m\mathbb{Z}/\{0\}$ and $%
\hat{T}_{0,u_{1}}$ is invertible in the orthogonal
complement to $\mbox{span}(\mathcal{J}_{k_{1}}a_{1})$ in $E_{1}${.} We
conjecture that for $\alpha =2$, the condition of being $2\pi /m$%
-nondegenerate holds for a generic set of central configurations in the set of
parameters. This condition is verified for the $k$-polygonal configuration
in Section 4 for $k=4,...,1000$.

\begin{lemma}
Assume $\alpha =2$. Under the conditions {\normalfont\textbf{(C0)-(C2)}}, if 
$a_{1}$ is $2\pi /m$-nondegenerate for $m\geq 2$, the statement of Lemma \ref%
{lem: bounded inverse} holds true after replacing $W$ by the fixed point set 
$W^{\Gamma }$.
\end{lemma}

\begin{proof}
By conditions \textbf{(C0)-(C1)}, 
\begin{equation*}
\hat{T}_{\ell }=\hat{T}_{\ell ,u_{0}}\oplus \hat{T}_{\ell ,u_{1}}
\end{equation*}%
because $u=(u_{0},u_{1})$. Since $a_{1}$ is $2\pi /m$-nondegenerate, the
matrix $\hat{T}_{\ell ,u_{1}}$ is invertible for the
Fourier modes $\ell =0,\pm m,\pm 2m,\dots $. A path $u=(u_{0},u_{1})$
belongs to $X^{\Gamma }$ if $u_{1}$ is $2\pi /m$-periodic. In particular,
the Fourier expansion of $u_{1}$ is fixed by $\Gamma $ only if 
\begin{equation*}
\hat{u}_{1,\ell }=0\text{ for }\ell \neq 0,\pm m,\pm 2m,\dots 
\end{equation*}%
and hence the operator $\partial _{\eta }F_{0}[(u_{a},0)]$ is invertible on $%
W^{\Gamma }$. The argument in the proof of Lemma \ref{lem: bounded inverse}
applies now in the fixed point space. Moreover, note that the group action
of $G$ commutes with that of $\Gamma $. Hence by $\mathbf{(C2)}$, the orbit $%
G(u_{a})$ belongs to $X_{0}^{\Gamma }\subset X^{\Gamma }$ and the functional 
$\mathcal{A}_{0}$ restricted to $X^{\Gamma }$ is still $G$-invariant.
\end{proof}

\subsection*{Lyapunov-Schmidt reduction}

Because of  Lemmas 3.1, 3.2 and 3.3, we can perform a Lyapunov-Schmidt
reduction as in Theorem 3.2, 3.3 and 3.4 in \cite{Braids}.

\begin{theorem}[Lyapunov-Schmidt reduction and uniform estimates ]
Under conditions {\normalfont\textbf{(A)-(B)}} and $\alpha \neq 2$, there is 
$\varepsilon _{0}>0$ such that for every $\varepsilon \in (0,\varepsilon
_{0})$, there is a $G$-invariant neighbourhood $\mathcal{V}\subset X_0$ of 
$u_{a}$ and an analytic $H$-equivariant mapping $\varphi _{\varepsilon }:%
\mathcal{V}\subset X_0 \rightarrow W$ such that $P_W\nabla \mathcal{A}(\xi +\eta )=0$ for $%
\xi \in \mathcal{V}$, $\xi+\eta\in X$ if and only if $\eta =\varphi _{\varepsilon }(\xi )$.
The system reduces to the finite-dimensional system $P_{X_0}\nabla \Psi_{\varepsilon}(\xi)=0$ in $\mathcal{V}$ where $\Psi _{\varepsilon }(\xi )=%
\mathcal{A}(\xi +\varphi _{\varepsilon }(\xi ))$, and $P_{X_0}$ is the projection of $X$ to $X_0$ composed with $P_X$. Furthermore, for each $\xi
\in \mathcal{V}$ and $g\in G$, the following estimate holds: 
\begin{equation}
\Vert \varphi _{\varepsilon }(\xi )\Vert \leq N_{1}(\varepsilon +\Vert \xi
-gu_{a}\Vert ^{2}).
\end{equation}%
If $\alpha =2$, the same result holds under the additional conditions {%
\normalfont\textbf{(C0)-(C2)}} after replacing $X_{0}$ and $W$ by their $%
\Gamma $-fixed point spaces.
\end{theorem}

We set $\xi =(\xi _{0},\xi _{1},\dots ,\xi _{n})$ where $\xi _{j}\in E_{j}$
are coordinates for the $k_{j}$-body problems, and $\xi _{0}\in E_{0}$ are
coordinates for the $n$-body problem formed by the centers of mass of the $%
n $ clusters. Once we fix $\varepsilon \in (0,\varepsilon _{0})$, the
function $\Psi _{\varepsilon }=\mathcal{V}\subset X_0\rightarrow \mathbb{R}$ is of the
form $\Psi _{\varepsilon }(\xi )=\mathcal{A}_{0}(\xi )+\mathcal{N}(\xi )$
where 
\begin{equation}
\mathcal{A}_{0}(\xi )=2\pi \left( V_{0}(\xi
_{0})+\sum_{j=1}^{n}r_{j}^{1-\alpha }V_{j}(\xi _{j})\right) .  \notag
\label{A0 reduced}
\end{equation}%
is $G$-invariant, and 
\begin{equation*}
\mathcal{N}(\xi )=\mathcal{A}_{0}(\xi +\varphi _{\varepsilon }(\xi ))-%
\mathcal{A}_{0}(\xi )+\mathcal{H}(\xi +\varphi _{\varepsilon }(\xi ))
\end{equation*}%
is $H$-invariant and satisfies the estimate

\begin{equation}
\Vert \mathcal{C}_{\varepsilon }\nabla \mathcal{N}(\xi )\Vert \leq
N(\varepsilon +\Vert \xi -gu_{a}\Vert ^{2}) , \label{estimate N}
\end{equation}%
where $\mathcal{C}_{\varepsilon }=\mathcal{I}_{n}\oplus r_{1}^{\alpha -1}%
\mathcal{I}_{k_{1}}\oplus \dots \oplus r_{n}^{\alpha -1}\mathcal{I}_{k_{n}}$ and $g\in G$%
. The proof of this estimate follows the same steps as in Theorem
3.5 in our previous work \cite{Braids}.

\subsection*{Symmetry reduction}

Even if the problem is now finite dimensional, it is still not possible to
continue the solutions of $P_{X_0}\nabla \Psi _{\varepsilon }(\xi )=0$ from $%
\varepsilon =0$ because $\Psi _{\varepsilon }$ still blows up when $%
\varepsilon \rightarrow 0$. We obtain a regular function by passing to the
quotient space under the action of $H$ on $\mathcal{V}$. Accordingly we
write 
\begin{equation}
\xi ^{\prime }=(\xi _{1},\dots ,\xi _{n})\in E^{\prime }=E_{1}\times \dots
\times E_{n}
\end{equation}%
and $a^{\prime }=(a_{1},\dots ,a_{n})\in E^{\prime }$. The gradient equation 
$P_{X_0}\nabla \Psi _{\varepsilon }(\xi _{0},\xi ^{\prime })=0$ splits into two
parts 
\begin{equation}
P_{E_0}\nabla\Psi _{\varepsilon }(\xi _{0},\xi ^{\prime })=0\quad %
\mbox{and}\quad P_{E^{\prime }}\nabla\Psi _{\varepsilon }(\xi _{0},\xi
^{\prime })=0.  \label{reduced eq}
\end{equation}%
We perform a second reduction to express the regular part $\xi _{0}$ with respect to the singular part $\xi ^{\prime }$. We solve  $\xi
_{0}(\xi ^{\prime },\varepsilon )$ from the equation $P_{E_0}\nabla \Psi
_{\varepsilon }(\xi _{0},\xi ^{\prime })=0$. 

The group $G=G_{0}\times
G^{\prime }$ with $G_{0}=U(1)$ and $G^{\prime }=U(1)^{n}$ acts diagonally on 
$E_{0}\times E^{\prime }$. We define the one-codimensional subspace of the regular part $E_{0}$, 
\begin{equation}
E_{0}^{\prime }=\left\{ \zeta _{0}=(\rho _{0,1}e^{i\theta _{0,1}},\dots
,\rho _{0,n}e^{i\theta _{0,n}})\mid \sum_{j=1}^{n}\theta _{0,j}=0\right\} .
\end{equation}%
For every $\xi _{0}\in E_{0}$ we can find $h\in H=\widetilde{U(1)}$
such that $\xi _{0}=h\zeta _{0}$ for some $\zeta _{0}\in E_{0}^{\prime }$.
Setting $\xi ^{\prime }=h\zeta ^{\prime }$ one uses $H$-invariance to get 
\begin{equation*}
\Psi _{\varepsilon }(\xi _{0},\xi ^{\prime })=\Psi _{\varepsilon }(h\zeta
_{0},h\zeta ^{\prime })=\Psi _{\varepsilon }(\zeta _{0},\zeta ^{\prime }).
\end{equation*}%
The function $\Psi _{\varepsilon }$ now only depends on $(\zeta _{0},\zeta
^{\prime })\in E_{0}^{\prime }\times E^{\prime }$ and \eqref{reduced eq}
become 
\begin{equation}
P_{E_0^{\prime}}\nabla \Psi _{\varepsilon }(\zeta _{0},\zeta ^{\prime })=0\quad %
\mbox{and}\quad P_{E^{\prime}}\nabla\Psi _{\varepsilon }(\zeta
_{0},\zeta ^{\prime })=0.  \label{reduced eq2}
\end{equation}

\begin{definition}
\label{Def3}The central configuration $a_{0}\in E_{0}$ is \textbf{%
nondegenerate} if the orbit $U(1)(a_{0})$ is a nondegenerate critical manifold of the
amended potential $\left.V_{0}\right|_{E_0}:E_{0}\rightarrow \mathbb{R}$ defined in %
\eqref{defn: amended potential}.
\end{definition}

In this case, the kernel of $P_{E_0}\left.\nabla _{u_{0}}^{2}V_{0}[a_{0}]\right|_{E_0}$ is generated
by $\mathcal{J}_{n}a_{0}$. This is
equivalent to the assumption that the matrix $\hat{T}_{0,u_{0}}$ in %
\eqref{G_0 blocks} is invertible in a complement to $\mbox{Span}(\mathcal{J}%
_{n}a_{0})$ in $E_0$.

\begin{theorem}
\label{main regular} Suppose that $a_{0}\in E_{0}$ is nondegenerate. Then
for $\varepsilon \in (0,\varepsilon _{0})$, the critical points of $\Psi
_{\varepsilon }(\zeta _{0},\zeta ^{\prime })$ in the (possibly smaller)
neighbourhood $\mathcal{V}$ are in one to one correspondence with the
critical points of the function $\Psi _{\varepsilon }^{\prime }:\mathcal{V}%
^{\prime }\subset E^{\prime }\rightarrow \mathbb{R}$ given by 
\begin{equation*}
\Psi _{\varepsilon }^{\prime }(\zeta ^{\prime })=\sum_{j=1}^{n}\left(
r_{j}/\varepsilon \right) ^{1-\alpha }V_{j}(\zeta _{j})+\mathcal{N}^{\prime
}(\zeta ^{\prime })\text{,}
\end{equation*}%
where $\mathcal{V}^{\prime }\subset E^{\prime}$ is a neighbourhood of the
orbit $G^{\prime }(a^{\prime })$, $V_{j}(\zeta _{j})$ is the amended
potential and 
\begin{equation}
\mathcal{N}^{\prime }(\zeta ^{\prime })=\varepsilon ^{\alpha -1}\left( \frac{%
1}{2\pi }\mathcal{A}_{0}(\zeta _{0}(\zeta ^{\prime },\varepsilon ),\zeta
^{\prime })-\sum_{j=1}^{n}r_{j}^{1-\alpha }V_{j}(\zeta _{j})\right) +\frac{%
\varepsilon ^{\alpha -1}}{2\pi }\mathcal{N}(\zeta _{0}(\zeta ^{\prime
},\varepsilon ),\zeta ^{\prime }),  \label{defn= N'}
\end{equation}%
where $\zeta _{0}(\cdot ,\varepsilon ):\mathcal{V}^{\prime
}\subset E^{\prime}\rightarrow \mathbb{R}$ is unique such that $P_{E_0^{\prime}}\nabla \Psi _{\varepsilon }(\zeta _{0}(\zeta ^{\prime },\varepsilon
),\zeta ^{\prime })=0.$
\end{theorem}

\begin{proof}
Consider the equations obtained in \eqref{reduced eq2} 
\begin{equation*}
P_{E_0^{\prime}}\nabla \Psi _{\varepsilon }(\zeta _{0},\zeta ^{\prime })=0\quad %
\mbox{and}\quad P_{E^{\prime}}\nabla\Psi _{\varepsilon }(\zeta
_{0},\zeta ^{\prime })=0. 
\end{equation*}%
The uniform estimate 
\begin{equation*}
\Vert P_{X_0}\mathcal{C}_{\varepsilon }\nabla \mathcal{N}\left( \zeta _{0},\zeta
^{\prime }\right) \Vert \leq\Vert \mathcal{C}_{\varepsilon }\nabla \mathcal{N}\left( \zeta _{0},\zeta
^{\prime }\right) \Vert \leq N\left( \Vert \zeta ^{\prime
}-g^{\prime}a^{\prime }\Vert ^{2}+\Vert \zeta _{0}-a_{0}\Vert
^{2}+\varepsilon \right)\quad\mbox{for each}\quad g'\in G'
\end{equation*}%
implies that $$\lim_{\varepsilon \rightarrow 0} P_{E_0^{\prime}}\mathcal{C}_{\varepsilon }\nabla\mathcal{N}%
(a_{0},g^{\prime }a^{\prime })=0$$ and $\lim_{\varepsilon \rightarrow
0}P_{E_0^{\prime}}\mathcal{C}_{\varepsilon }\nabla ^{2}\mathcal{N}[a_{0},g^{\prime }a^{\prime }]=0$ for
each $g^{\prime }\in G'$. This is because the scaling matrix $\mathcal{C}_{\varepsilon }$ acts as the identity on the component $\zeta_0$. In particular, $P_{E_0^{\prime}}\nabla\Psi
_{0}( a_{0},g^{\prime }a^{\prime })=0$ and the Hessian 
\begin{equation*}
P_{E_0^{\prime}}\left.\nabla^{2}\Psi _{0}[a_{0},g^{\prime }a]\right|_{E_0^{\prime}}=P_{E_0^{\prime}}\left.\nabla^{2}V_{0}[a_{0}]\right|_{E_0^{\prime}}
\end{equation*}%
is non-singular on $E_0^{\prime}$ by the assumption that $a_{0}$ is nondegenerate. Using the
implicit function theorem and the compactness of $G^{\prime }$, there is,
for each $\varepsilon $ sufficiently small, a smooth function $\zeta
_{0}(\zeta ^{\prime },\varepsilon )$ defined on a neighbourhood $\mathcal{V}%
^{\prime }\subset E^{\prime }$ of the orbit $G^{\prime }(a^{\prime })$ such
that 
\begin{equation}
P_{E_0^{\prime}}\nabla\Psi _{\varepsilon }(\zeta _{0}(\zeta ^{\prime
},\varepsilon ),\zeta ^{\prime })=0  \label{zeta0 equation}
\end{equation}%
on this neighbourhood and $\zeta _{0}(a^{\prime },0)=a_{0}$. Hence, when we
fix $\varepsilon \in (0,\varepsilon _{0})$ with $\varepsilon _{0}$ possibly
smaller, and take a smaller neighbourhood $\mathcal{V}\subset
E^{\prime}_{0}\times E^{\prime }$, the critical points of $\Psi
_{\varepsilon }(\zeta _{0},\zeta ^{\prime })$ in $\mathcal{V}$ are in one to
one correspondence with the critical points of the function $\Psi
_{\varepsilon }^{\prime }:\mathcal{V}^{\prime }\subset E^{\prime
}\rightarrow \mathbb{R}$ given by 
\begin{equation*}
\Psi _{\varepsilon }^{\prime }(\zeta ^{\prime }):=\frac{\varepsilon ^{\alpha
-1}}{2\pi }\left( \mathcal{A}_{0}(\zeta _{0}(\zeta ^{\prime },\varepsilon
),\zeta ^{\prime })+\mathcal{N}(\zeta _{0}(\zeta ^{\prime },\varepsilon
),\zeta ^{\prime })\right) =\sum_{j=0}^{n}\left( r_{j}/\varepsilon \right)
^{1-\alpha }V_{j}(\zeta _{j})+\mathcal{N}^{\prime }(\zeta ^{\prime }).
\end{equation*}%
These are solutions of the equation $P_{E^{\prime}}\nabla\Psi _{\varepsilon }^{\prime }(\zeta ^{\prime })=0$.
Note that the regularizing factor of $\varepsilon ^{\alpha -1}/2\pi $ leaves
the equation \eqref{zeta0 equation} unchanged when we fix $\varepsilon>0$.
\end{proof}

\begin{lemma}
\label{lemma: estimate N'} There are constants $N^{\prime },N^{\prime \prime
}>0$ such that, for each $g^{\prime }\in G'$, 
\begin{equation*}
\Vert \zeta _{0}(\zeta ^{\prime },\varepsilon )- a_{0}\Vert \leq N^{\prime
}(\varepsilon +\Vert \zeta ^{\prime }-g^{\prime }a^{\prime }\Vert ^{2})\quad %
\mbox{and}\quad \Vert P_{E^{\prime}}\nabla_{\zeta
^{\prime }}\mathcal{N}(\zeta ^{\prime })\Vert \leq N^{\prime \prime }(\varepsilon +\Vert \zeta ^{\prime
}-g^{\prime }a^{\prime }\Vert ^{2}).
\end{equation*}
\end{lemma}

\begin{proof}
We first write the Taylor expansion of the operator $$P_{E_0^{\prime}}\nabla 
\mathcal{A}_{0}(\zeta _{0}(\zeta ^{\prime },\varepsilon ),\zeta ^{\prime
})=2\pi P_{E_0^{\prime}}\nabla V_{0}(\zeta _{0}(\zeta ^{\prime },\varepsilon ))$$
around $\zeta _{0}=a_{0}$. For simplicity, we omit the dependence of $%
\varepsilon $ in the function $\zeta _{0}(\zeta ^{\prime },\varepsilon )$.
We may shrink the neighbourhood $\mathcal{V}^{\prime }$ such that for each $%
\zeta ^{\prime }\in \mathcal{V}^{\prime }$ we get 
\begin{equation*}
\Vert P_{E_0^{\prime}}\nabla V_{0}(\zeta _{0}(\zeta ^{\prime }))+P_{E_0^{\prime}}\nabla^{2}V_{0}[a_{0}](\zeta _{0}(\zeta ^{\prime })-a_{0})\Vert \leq
N_{1}^{\prime }\Vert \zeta _{0}(\zeta ^{\prime })-a_{0}\Vert ^{2}
\end{equation*}%
for some constant $N_{1}^{\prime }$. By the reverse triangle inequality, 
\begin{equation}
\Vert P_{E_0^{\prime}}\nabla^{2}V_{0}[a_{0}](\zeta _{0}(\zeta ^{\prime
})-a_{0})\Vert \leq \Vert P_{E_0^{\prime}}\nabla V_{0}(\zeta _{0}(\zeta
^{\prime }))\Vert +N_{1}^{\prime }\Vert \zeta _{0}(\zeta ^{\prime
})-a_{0}\Vert ^{2}.  \label{lem36 eq1}
\end{equation}%
Since $\zeta _{0}(\zeta ^{\prime })$ solves
uniquely the equation $P_{E_0^{\prime}}\nabla\Psi _{\varepsilon }(\zeta _{0}(\zeta ^{\prime
}),\zeta ^{\prime })=0$, we have 
\begin{equation*}
2\pi P_{E_0^{\prime}}\nabla V_{0}(\zeta _{0}(\zeta ^{\prime }))=-P_{E_0^{\prime}}\nabla\mathcal{N}(\zeta _{0}(\zeta ^{\prime }),\zeta ^{\prime }).
\end{equation*}%
The estimates for $\mathcal{N}$ in \eqref{estimate N} yields 
\begin{equation*}
2\pi \Vert P_{E_0^{\prime}}\nabla V_{0}(\zeta _{0}(\zeta ^{\prime })\Vert
^{2}\leq \Vert \nabla\mathcal{N}(\zeta _{0}(\zeta ^{\prime }),\zeta ^{\prime })\Vert^2\leq  N\left( \varepsilon +\Vert \zeta _{0}(\zeta ^{\prime
})-a_{0}\Vert ^{2}+\Vert \zeta ^{\prime }-a^{\prime }\Vert ^{2}\right)
\end{equation*}%
and \eqref{lem36 eq1} becomes%
\begin{equation*}
\Vert P_{E_0^{\prime}}\nabla^{2}V_{0}[a_{0}](\zeta _{0}(\zeta ^{\prime
})-a_{0})\Vert \leq N_{2}^{\prime }(\varepsilon +\Vert \zeta _{0}(\zeta
^{\prime })-a_{0}\Vert ^{2}+\Vert \zeta ^{\prime }-a^{\prime }\Vert ^{2})
\end{equation*}%
for some $N_{2}^{\prime }$. Since $P_{E_0^{\prime}}\left.\nabla^{2}V_{0}[a_{0}]\right|_{E_0^{\prime}}$ is
invertible, there is $c>0$ such that $\Vert P_{E_0^{\prime}}\left.\nabla^{2}V_{0}[a_{0}]\right|_{E_0^{\prime}}\Vert \geq 2c$. Therefore 
\begin{equation*}
\Vert \zeta _{0}(\zeta ^{\prime })-a_{0}\Vert (2c-N_{2}^{\prime }\Vert \zeta
_{0}(\zeta ^{\prime })-a_{0}\Vert )\leq N_{2}^{\prime }(\varepsilon +\Vert
\zeta ^{\prime }-a^{\prime }\Vert ^{2}).
\end{equation*}%
We can take $\mathcal{V}^{\prime }$ smaller such that $\Vert \zeta
_{0}(\zeta ^{\prime })-a_{0}\Vert \leq c/N_{2}^{\prime }$ and then 
\begin{equation*}
c\Vert \zeta _{0}(\zeta ^{\prime })-a_{0}\Vert \leq N_{2}^{\prime
}(\varepsilon +\Vert \zeta ^{\prime }-a^{\prime }\Vert ^{2})~.
\end{equation*}%
We now set $N^{\prime }=N_{2}^{\prime }/c$. This allows us to write $\zeta
_{0}(\zeta ^{\prime })=a_{0}+R_{\varepsilon }(\zeta ^{\prime })$, where $%
R_{\varepsilon }(\zeta ^{\prime })$ is the remainder which satisfies by the
above 
\begin{equation}
\Vert R_{\varepsilon }(\zeta ^{\prime })\Vert \leq N^{\prime }(\varepsilon
+\Vert \zeta ^{\prime }-a^{\prime }\Vert ^{2}).  \label{estimate R}
\end{equation}

To obtain the second estimate, we use the definition \eqref{defn= N'} and we
replace $\zeta _{0}(\zeta ^{\prime })$ by $a_{0}+R_{\varepsilon }(\zeta
^{\prime })$. By \eqref{A0 reduced}, the first terms of \eqref{defn= N'}
become 
\begin{equation*}
\frac{\varepsilon ^{\alpha -1}}{2\pi }\mathcal{A}_{0}(a_{0}+R_{\varepsilon
}(\zeta ^{\prime }),\zeta ^{\prime })-\sum_{j=1}^{n}\left( r_{j}/\varepsilon
\right) ^{1-\alpha }V_{j}(\zeta _{j})=\varepsilon ^{\alpha
-1}V_{0}(a_{0}+R_{\varepsilon }(\zeta ^{\prime })).
\end{equation*}%
Applying the mean value theorem, there is some $\mu \in \lbrack 0,1]$ such
that%
\begin{equation*}
P_{E^{\prime}}\nabla V_{0}(a_{0}+R_{\varepsilon }(\zeta ^{\prime
}))=P_{E^{\prime}}\nabla ^{2}V_{0}(a_{0}+\mu R_{\varepsilon }(\zeta
^{\prime }))\left( R_{\varepsilon }(\zeta ^{\prime })\right) ,
\end{equation*}%
and there is a constant $e$ such that 
\begin{equation}
\left\Vert P_{E^{\prime}}\nabla V_{0}(a_{0}+R_{\varepsilon }(\zeta
^{\prime }))\right\Vert \leq e\left\Vert R_{\varepsilon }(\zeta ^{\prime
})\right\Vert .  \label{in}
\end{equation}%
We finally get 
\begin{equation*}
\Vert P_{E^{\prime}}\nabla \mathcal{N}(\zeta ^{\prime })\Vert \leq
\varepsilon ^{\alpha -1}e\left\Vert R(\zeta ^{\prime })\right\Vert
+N(\varepsilon +\Vert \zeta ^{\prime }-a\Vert ^{2}+\Vert R_{\varepsilon
}(\zeta ^{\prime })\Vert ^{2})\leq N^{\prime \prime }(\varepsilon +\Vert
\zeta ^{\prime }-a^{\prime }\Vert ^{2})
\end{equation*}%
for some $N^{\prime \prime }$ after using the uniform estimate 
\eqref{estimate
N} and \eqref{estimate R}. The arguments can be repeated replacing $%
a^{\prime }$ by $g^{\prime }a^{\prime }$ for each $g^{\prime }\in G^{\prime
} $ and the neighbourhood $\mathcal{V}^{\prime }$ can be taken as a
neighbourhood of the orbit $G^{\prime }(a^{\prime })$ by compactness of $%
G^{\prime }$.
\end{proof}

\subsection*{Lyusternik-Schnirelmann application}

We now show that the function 
\begin{equation}
\Psi _{\varepsilon }^{\prime }(\zeta ^{\prime })=\sum_{j=1}^{n}\left(
r_{j}/\varepsilon \right) ^{1-\alpha }V_{j}(\zeta _{j})+\mathcal{N}^{\prime
}(\zeta ^{\prime })\text{,}  \label{new reduced potential}
\end{equation}%
has critical points in the neighbourhood $\mathcal{V}^{\prime }\subset
E^{\prime }$ of the orbit $G^{\prime }(a^{\prime })$.

\begin{theorem}
\label{main category}If $a_{0}$ is a nondegenerate central configuration and 
$a_{j}$ is a $2\pi p_j$-nondegenerate central configuration for each $%
j=1,\dots,n$. Then, for each $\varepsilon \in (0,\varepsilon _{0})$, there
is a neighbourhood $\mathcal{V}^{\prime }\subset E^{\prime }$ of the orbit $%
G^{\prime }(a^{\prime })$ such that the number of critical points of the
function $\Psi^{\prime} _{\varepsilon }:\mathcal{V}^{\prime }\to \mathbb{R}$
is bounded below by $\mbox{Cat}(G^{\prime }/K)$ where $K$ is the stabiliser
of $a^{\prime }$.
\end{theorem}

\begin{proof}
Notice that to critical points of the function $\Psi _{\varepsilon }^{\prime
}(\zeta ^{\prime })$ restricted to the subspace $E^{\prime }\subset E^{N}$
are the solutions of $P_{E^{\prime }}\nabla \Psi _{\varepsilon }^{\prime
}(\zeta ^{\prime })=0$. Using the expansion (\ref{expansion}), we obtain $%
\lim_{\varepsilon \rightarrow 0}\left( r_{j}/\varepsilon \right) ^{1-\alpha
}=p_{j}^{2(\alpha -1)/(\alpha +1)}$, and 
\begin{equation*}
P_{E_{j}}\nabla \Psi _{0}^{\prime }(a^{\prime })=\left( p_{j}^{2(\alpha
-1)/(\alpha +1)}P_{E_{j}}\nabla V_{j}(a_{j})+\lim_{\varepsilon \rightarrow
0}P_{E_{j}}\nabla \mathcal{N}^{\prime }(a^{\prime })\right) .
\end{equation*}%
By the second estimate in Lemma \ref{lemma: estimate N'}, $\Vert
\lim_{\varepsilon \rightarrow 0}P_{E^{\prime }}\nabla _{\zeta ^{\prime }}%
\mathcal{N}^{\prime }(\zeta ^{\prime })\Vert \leq N^{\prime }\Vert \zeta
^{\prime }-a^{\prime }\Vert ^{2}$, we have 
\begin{equation*}
\lim_{\varepsilon \rightarrow 0}P_{E^{\prime }}\nabla \mathcal{N}^{\prime
}(a^{\prime })=0,\qquad \lim_{\varepsilon \rightarrow 0}P_{E^{\prime
}}\nabla ^{2}\mathcal{N}^{\prime }[a^{\prime }]|_{E^{\prime }}=0.
\end{equation*}
Since $a_{j}$ is a critical point of $V_{j}$ for each $j=1,\dots ,n$, we get 
$P_{E_{j}}\nabla V_{j}(a_{j})$ and $P_{E^{\prime }}\nabla \Psi _{0}^{\prime
}(a^{\prime })=0$. Furthermore, the same estimate implies that the Hessian
of $\Psi _{0}^{\prime }$ respect to $E^{\prime }$ is 
\begin{equation*}
P_{E^{\prime }}\nabla ^{2}\Psi _{0}^{\prime }[a^{\prime
}]|_{E^{\prime }}=p_{1}^{2(\alpha -1)/(\alpha +1)}\hat{T}_{0,u_{1}}\oplus \dots
\oplus p_{n}^{2(\alpha -1)/(\alpha +1)}\hat{T}_{0,u_{n}},
\end{equation*}%
where 
\begin{equation*}
\hat{T}_{0,u_{j}}=P_{E_{j}}\nabla ^{2}V_{j}[a_{j}]|_{E_{j}}\in \mbox{End}%
(E_{j}).
\end{equation*}%
By the $2\pi p_{j}$-nondegeneracy assumption on $a_{j}$, each block $\hat{T}%
_{0,u_{j}}$ is non-singular on a subspace $W_{j}$ complementary to $%
T_{a_{j}}U(1)(a_{j})$ in $E_{j}$. Consequently, $W=\bigoplus_{j=1}^{n}W_{j}$
is a complement to $T_{a^{\prime }}G^{\prime }(a^{\prime })$ in $E^{\prime }$%
. The argument above is valid if we replace $a^{\prime }$ by $g^{\prime
}a^{\prime }$ for any $g^{\prime }\in G^{\prime }$. A standard application
of the Palais-Slice coordinates as in \cite{Braids} allows us to express the
normal coordinates in $W$ in term of the coordinates along the group orbit $%
G^{\prime }(a^{\prime })$ after taking $\varepsilon _{0}$ and $\mathcal{V}%
^{\prime }$ possibly smaller. For $\varepsilon \in (0,\varepsilon _{0})$,
the solutions of $P_{E^{\prime }}\nabla \Psi _{\varepsilon }^{\prime }(\zeta
^{\prime })=0$ in $\mathcal{V}^{\prime }$ are in one to one correspondence
with the critical points the function $\Psi _{\varepsilon }^{\prime
}:G^{\prime }(a^{\prime })\rightarrow \mathbb{R}$. By the
Lyusternik-Schnirelmann theorem for compact manifolds \cite{Fadell, LS}, the
number of critical points of this function is bounded below by $\mbox{Cat}%
(G^{\prime }/K)$ where $K$ is the stabiliser of $a^{\prime }$.
\end{proof}

\subsection*{Existence of carousel solutions}

We state our main result regarding the existence of carousel solutions of
the $N$-body problem. In summary, we proved that the existence of periodic
solutions near the orbit $G(u_{a})$ with $u_{a}=(a_{0},a_{1},%
\dots ,a_{n})$ reduces to determining whether the function 
\eqref{new
reduced potential} admits critical points.  Recall that we
assume that $k_{j}>1$ for $j=1,\dots ,n_{0}$ and $k_{j}=1$ for $%
j=n_{0}+1,\dots ,n$. Thus $K=U(1)^{n-n_{0}}$, because $G^{\prime
}=U(1)^{n}$ acts trivially on each subspace $E_{j}=\{0\}$ for $%
j=n_{0}+1,\dots ,n$. By Theorem \ref{main category},
there are at least $\mbox{Cat}(G^{\prime }/K)=n_{0}+1$ solutions. Therefore, we have,

\begin{theorem}[Carousels for non-gravitational potentials]
\label{main result}Set $\alpha \neq 2$ and $k_{j}>1$ for $j=1,\dots ,n_{0}$
and $k_{j}=1$ for $j=n_{0}+1,\dots ,n$. Fix integers $p_{1},\dots ,p_{n_0}\in 
\mathbb{Z}\backslash \{0\}$ and choose the frequencies $\omega _{j},\nu $
and the amplitudes $r_{j}$ according to the conditions {\normalfont\textbf{%
(A)-(B)}}. Suppose that $a_{0}$ is a nondegenerate central configuration of
the $n$-body problem and $a_{j}$ is a $2\pi p_{j}$-nondegenerate central
configuration of the $k_{j}$-body problem for $j=1,\dots ,n_{0}$. Then for
every sufficiently small $\varepsilon $, there are at least $\mbox{Cat}(G^{\prime }/K)=n_{0}+1$ solutions  of the $N$-body problem \eqref{NBP}
with components of the form%
\begin{eqnarray*}
q_{j,k}(t) &=&\exp (tJ)u_{0,j}(\nu t)+r_{j}\exp (t\omega _{j}J)u_{j,k}(\nu
t),\quad j=1,\dots ,n_{0}, \quad k=1, \dots, k_j\\
q_{j,1}(t) &=&\exp (tJ)u_{0,j}(\nu t)\quad \mbox{for}\;j=n_{0}+1,\dots ,n.
\end{eqnarray*}%
where $u_{0,j}(\nu t)=a_{0,j}+\mathcal{O}_{X}(\varepsilon )$ and $%
u_{j,k}(\nu t)=e^{\vartheta _{j}J}a_{j,k}+\mathcal{O}_{X}(\varepsilon )$ for 
$j=1,\dots ,n_{0}$ with some phases $\vartheta _{j}\in S^{1}$.
\end{theorem}

These solutions are quasi-periodic if $\nu \notin \mathbb{Q}$, and periodic
if $\nu \in \mathbb{Q}$. In the case that $\nu =(p-q)/q$ is rational, then $%
\varepsilon ^{-\left( \alpha +1\right) /2}=p/q$ and $\omega
_{j}=(q+p_{j}p)/q $ are rational. Thus, for any fixed integer $q>0$, there
is some $p_{0}>0$ such that, for each $p>p_{0}$ , the components $q_j(t)$ are $2\pi q$-periodic. In these solutions the centers of mass
of $n$ clusters (close to the central configuration $a_{0}$) wind around the
origin $q$ times, while each $j$-cluster winds around its center
of masses $q+p_{j}p$ times. The sign of the frequency $\omega _{j}$ is determined by $p_{j}$ and represents
whether the $j$-cluster has prograde or retrograde rotation with respect to the whole system. In the case of rational $\nu =\left( p-q\right) /q$, the \emph{%
prograde} ($p_{j}>0$) refers to the case that the cluster rotates in the
same direction as the main relative equilibrium, while \emph{retrograde} ($%
p_{j}<0$) refers to the case that the cluster rotates in the opposite
direction.

The strategy still applies for the gravitational potential under extra
assumptions. In this case, we have from condition {\normalfont\textbf{(C1)}}
that 
\begin{equation*}
r_{1}=\varepsilon ,\qquad \omega _{1}=\varepsilon ^{-\left( \alpha +1\right)
/2},\qquad \nu =\varepsilon ^{-\left( \alpha +1\right) /2}-1.
\end{equation*}

\begin{theorem}[Carousels for gravitational potentials]
\label{main result2}Set $\alpha =2$, assume conditions {\normalfont\textbf{%
(C0)-(C2)}} and choose the frequencies $\omega _{1},\nu $ and the amplitude $%
r_{1}$ according to the conditions {\normalfont\textbf{(C1)}}. Suppose that $%
a_{0}$ is a nondegenerate central configuration of the $n$-body problem, and 
$a_{1}$ is a $2\pi /m$-nondegenerate central configuration. Then for every
sufficiently small $\varepsilon $, there are at least $\mbox{Cat}(G^{\prime
}/K)=2$ solutions of the $N$-body problem with components of the form%
\begin{eqnarray*}
q_{1,k}(t) &=&\exp (tJ)u_{0,1}(\nu t)+r_{1}\exp (t\omega _{1}J)u_{1,k}(\nu
t), \\
q_{j,1}(t) &=&\exp (tJ)u_{0,j}(\nu t)\quad \mbox{for}\;j=2,\dots ,n.
\end{eqnarray*}%
where $u_{0,j}(\nu t)=a_{0,j}+\mathcal{O}_{X}(\varepsilon )$ and $%
u_{1,k}(\nu t)=e^{\vartheta _{1}J}a_{1,k}+\mathcal{O}_{X}(\varepsilon )$ for
some phase $\vartheta _{1}\in S^{1}$. Furthermore, each $u_{1,k}(s)$ is $%
2\pi /m$-periodic with $m\geq 2$ and 
\begin{equation}
u_{0,j}(s)=\exp (-\theta J)u_{0,\sigma (j)}(s+\theta )
\end{equation}%
where $(\theta ,\sigma )$ is the generator of the discrete symmetry group $%
\Gamma $ introduced in {\normalfont\textbf{(C0)}}.
\end{theorem}

\begin{example}
\label{exam}
The particular case $k_{1}=2$ and $k_{2},\dots ,k_{n}=1$ is studied in \cite%
{Braids} in the general setting $E=\mathbb{R}^{2d}$.
This follows from the fact that the Euler-Lagrange equations 
\eqref{EL: v
tilde} for $j=1$ are equivalent to a $2$-body problem in a rotating frame of
frequency $\omega $. That is, the system for $u_{1}\in E_{1}$ can be parametrized by $w_{1}\in E$ as 
\begin{equation*}
E_{1}=\left\{ u_{1}=(u_{1,1},u_{1,2})\in E^{2}:u_{1,1}=\frac{m_{1,2}}{%
m_{1,1}+m_{1,2}}w_{1},u_{1,2}=-\frac{m_{1,1}}{m_{1,1}+m_{1,2}}w_{1}\right\} .
\end{equation*}%
In these coordinates $\left( u_{0},w_{1}\right) \in E_{0}\times E$ the
action $\mathcal{A}_{0}(u_{0},u_{1})$ becomes 
\begin{equation*}
\mathcal{A}_{0}(u_{0},w_{1})=\int_{0}^{2\pi }\left( \mathcal{L}_{0}(u_{0},\dot{u}%
_{0})+r_{1}^{1-\alpha }\mathcal{L}_{1}(w_{1},\dot{w}_{1})\right) ,
\end{equation*}%
where 
\begin{equation*}
\mathcal{L}_{1}(w_{1},\dot{w}_{1})=\frac{m_{1,1}m_{1,2}}{m_{1,1}+m_{1,2}}\left\Vert
(\partial _{t}+\omega J)w_{1}\right\Vert ^{2}+m_{1,1}m_{1,2}\phi _{\alpha
}(\left\Vert w_{1}\right\Vert ),
\end{equation*}%
which corresponds to the Lagrangian of the Kepler problem in a
rotating frame of frequency $\omega $. We are then left with a
Kepler problem for $\mathcal{L}_{1}(w_{1},\dot{w}_{1})$ and an $n$-body problem for $%
\mathcal{L}_{0}(u_{0},\dot{u}_{0})$. We can normalize the total mass of the cluster by
setting $m_{1,1}+m_{1,2}=1$. The solutions \eqref{solutions} are of the
form%
\begin{align*}
q_{1,1}(t)& =\exp (tJ)u_{0,1}(\nu t)+\varepsilon m_{1,2}\exp (\omega
_{1}tJ)w_{1}(\nu t), \\
q_{1,2}(t)& =\exp (tJ)u_{0,1}(\nu t)-\varepsilon m_{1,1}\exp (\omega
_{1}tJ)w_{1}(\nu t), \\
q_{j,1}(t)& =\exp (tJ)u_{0,j}(\nu t),\quad j=2,\dots ,n,
\end{align*}%
which are exactly the solutions obtained in \cite{Braids} for the
particular case $E=\mathbb{R}^{2}$.
\end{example}

\section{The \texorpdfstring{$2\pi p$} --nondegeneracy property of the \texorpdfstring{$k$}--polygon}

In this section we verify the $2\pi p$-nondegeneracy property of a polygonal
central configuration (with frequency one) for $k$ bodies with masses equal
to one. We shall denote by $E_{red}^{k}$ the subspace of $E^{k}$ of central
configurations with center of mass fixed at the origin. With our previous
notations, we have $E_{j}=E_{red}^{k_{j}}$. By \eqref{T block}, a central
configuration $a\in E_{red}^{k}$ is $2\pi p$-nondegeneracy if the block $%
\hat{T}_{\ell ,u}\in \mbox{End}((E_{red}^{k})^{\mathbb{C}})$ given by $$\hat{T%
}_{\ell ,u}=(\ell ^{2}+1)^{-1}P_{E_{red}^{k}}M_{a}(\ell
/p)|_{E_{red}^{k}}\in \mbox{End}((E_{red}^{k})^{\mathbb{C}})$$ is a
non-singular matrix. Here 
\begin{equation}
M_{a}(\lambda )=\lambda ^{2}\mathcal{I}-2i\lambda \mathcal{J}+\nabla
_{u}^{2}V[a]\in \mbox{End}((E^{k})^{\mathbb{C}})  \label{Ma}
\end{equation}%
and $V$ is the amended potential%
\begin{equation}
V(u)=\frac{1}{2}\sum_{j=1}^{k}\left\Vert u_{j}\right\Vert ^{2}+\sum_{1\leq
i<j\leq k}\phi _{\alpha }\left( \left\Vert u_{j}-u_{i}\right\Vert \right) 
\text{.}  \label{section4 amended potential}
\end{equation}%
Specifically, the $2\pi p$-nondegeneracy property of $a$ is equivalent to
the conditions:

\begin{enumerate}
\item[(a)] $\hat{T}_{\ell ,u}$ is invertible for all $\ell \neq 0$,

\item[(b)] $\hat{T}_{0,u}$ has a one dimensional kernel generated by $%
\mathcal{J}_{k}a$.
\end{enumerate}

\subsection{Spectrum of the polygonal configuration}

First we find the ratio of the polygonal relative equilibrium with frequency one.

\begin{proposition}
The polygonal configuration
\begin{equation}
a=(s_{1})^{\frac{1}{\alpha+1}}(\exp(J\zeta)e_{1},\dots,\exp(kJ\zeta
)e_{1}),\quad e_{1}=%
\begin{bmatrix}
1\\
0
\end{bmatrix}
\label{polygonal configuration}%
\end{equation}
is a central configuration of frequency one, where $\zeta=2\pi/k$ and
\[
s_{1}=\frac{1}{2^{\alpha}}\sum_{j=1}^{k-1}\frac{1}{\sin^{\alpha-1}(j\zeta/2)}.
\]

\end{proposition}

\begin{proof}
We shall show that $a$ is a critical point of the amended potential $V$.
Consider the function
\begin{equation}
\widetilde{V}(u;\omega)=\frac{\omega}{2}\sum_{j=1}^{k}\left\Vert
u_{j}\right\Vert ^{2}+\sum_{1\leq i<j\leq k}\phi_{\alpha}\left(  \left\Vert
u_{j}-u_{i}\right\Vert \right)  .\label{V}%
\end{equation}
It satisfies $\widetilde{V}(u;1)=V(u)$ and has the scaling property
\begin{equation}
\nabla\widetilde{V}(ru;r^{-\left(  \alpha+1\right)  }\omega)=r^{-\alpha}%
\nabla\widetilde{V}(u;\omega)\quad\mbox{for any}\quad
r>0.\label{scaling property}%
\end{equation}
Thus $(\widetilde{a};\omega)$ is critical point of $\widetilde{V}$ if and only
if $(r\widetilde{a};r^{-\left(  \alpha+1\right)  }\omega)$ is a critical
points of $\widetilde{V}$ for any $r>0$. By \cite{GaIz11}, the unitary
polygon
\[
\widetilde{a}=(\exp(J\zeta)e_{1},\dots,\exp(kJ\zeta)e_{1}),\quad e_{1}=%
\begin{bmatrix}
1\\
0
\end{bmatrix}
\]
with $\omega=s_{1}$ is a critical point of $\widetilde{V}$. It follows that
$(a;1)$ where $a=(s_{1})^{\frac{1}{\alpha+1}}\widetilde{a}$ is a critical
point of $\widetilde{V}(u;1)=V(u)$. In other words, $a$ is a central
configuration with frequency one.
\end{proof}

\subsubsection*{Block diagonalisation}

Let $S_{k}$ be the permutation group of $k$ letters. The group $G=S_{k}\times
SO(2)$ acts on $E^{k}$ by
\[
(\sigma,\theta)\cdot(u_{1},\dots,u_{k})=\left(  \exp(-J\theta)u_{\sigma
(1)},\dots,\exp(-J\theta)u_{\sigma(k)}\right)  .
\]
The amended potential \eqref{section4 amended potential} is $G$-invariant. Let
$a\in E^{k}$ be the polygonal configuration \eqref{polygonal configuration}.
Its stabiliser is the subgroup $G_{a}=C_{k}$ generated by the element
$(\sigma,\zeta)\in G$ where $\sigma=(1\,2\,\dots\,k)$ and $\zeta=\frac{2\pi
}{k}$. The $G_{a}$-equivariant property of the Hessian $\nabla_{u}^{2}V[a]$ is
used in Proposition 7 of \cite{GaIz11} to find the irreducible representations
of $(E^{\mathbb{C}})^{k}$. By Schur's lemma the Hessian of $V$ is equivalent
to a block diagonal matrix with $k$ blocks corresponding to the isotypic components.

\begin{definition}
For $j=1,\dots,k$, we define isomorphisms $T_{j}:E^{\mathbb{C}}\rightarrow
W_{j}$ by%
\begin{equation}
T_{j}(w)=\frac{1}{\sqrt{k}}(\exp((ijI+J)\zeta)w,\dots,\exp(k(ijI+J)\zeta
)w),\label{Tmu}%
\end{equation}
where
\[
W_{j}=\left\{  (\exp((ijI+J)\zeta)w,\dots,\exp(k(ijI+J)\zeta)w)\mid w\in
E^{\mathbb{C}}\right\}  \subset(E^{\mathbb{C}})^{k}.
\]

\end{definition}

Specifically, in \cite{GaIz11} is proved that the subspaces $W_{j}$ are the
\defn{isotypic components} under the action of $G_{a}$. The group $G_{a}$ acts
on each subspace $W_{j}$ by rotating each component by $\exp(ij\zeta J)$.
Since the subspaces $W_{j}$ are mutually orthogonal, the endomorphism
$P\in\mbox{End}((E^{\mathbb{C}})^{k})$ defined by
\[
P(w_{1},\dots,w_{k})=\sum_{j=1}^{k}T_{j}(w_{j})
\]
is orthogonal. Since $P$ rearranges the coordinates of the isotypic
decomposition, it follows by Schur's Lemma that
\begin{equation}
P^{-1}\nabla_{u}^{2}V[a]P=B_{1}\oplus\dots\oplus B_{k}\label{diaghess}%
\end{equation}
where each $B_{j}\in\mbox{End}(E^{\mathbb{C}})$ satisfies
\begin{equation}
\nabla_{u}^{2}V[a]T_{j}(w)=T_{j}(B_{j}w).\label{eigen}%
\end{equation}

Define
\begin{equation}
s_{j}=\frac{1}{2^{\alpha}}\sum_{l=1}^{k-1}\frac{\sin^{2}(jl\zeta/2)}%
{\sin^{\alpha+1}(l\zeta/2)},\qquad\zeta=\frac{2\pi}{k}\text{.}%
\end{equation}
Following \cite{GaIz11}, the numbers $s_{j}$ have the following properties:
They are $k$-periodic, that is $s_{k+j}=s_{j}$.
They are symmetric with respect to $[k/2]$, that is $s_{[k/2]-j}=s_{[k/2]-j}$.
They increase as
$j$ increases, that is
\[
s_{j+1}>s_{j},\qquad 0\leq j\leq n/2\text{.}%
\]
In particular, $s_{j}>s_{0}=0$ for
$j=1,...,[n/2]$ with its maximum attained at $j=[n/2].$

\begin{proposition}
[Normal form of the amended potential]\label{prop: normal form} Each
endomorphism $B_{j}$ is a matrix of the form%
\begin{equation}
B_{j}=(1+\alpha_{j})I-\beta_{j}R-\gamma_{j}iJ\label{Bell}%
\end{equation}
where $I$ is the identity matrix,
\[
R=%
\begin{pmatrix}
1 & 0\\
0 & -1
\end{pmatrix}
\qquad J=%
\begin{pmatrix}
0 & -1\\
1 & 0
\end{pmatrix}
\]
and the coefficients are given by
\begin{equation}
\alpha_{j}=\frac{\alpha-1}{4s_{1}}(s_{j+1}+s_{j-1}),\qquad\beta_{j}%
=\frac{\alpha+1}{2s_{1}}(s_{j}-s_{1}),\qquad\gamma_{j}=\frac{\alpha-1}{4s_{1}%
}(s_{j+1}-s_{j-1}).\label{coeff}%
\end{equation}
\label{EnBk2}
\end{proposition}

\begin{proof}
By the scaling property \eqref{scaling property} of the amended potential
$\widetilde{V}$ given in \eqref{V}, we have%
\begin{equation}
\nabla_{u}^{2}V[a]=\nabla_{u}^{2}\widetilde{V}[(a\mathbf{;}1)]=\frac{1}{s_{1}%
}\nabla_{u}^{2}\widetilde{V}[(\widetilde{a}\mathbf{;}s_{1})].\label{scaling}%
\end{equation}
By \eqref{eigen} we get
\[
\nabla_{u}^{2}V[a]T_{j}(w)=\frac{1}{s_{1}}\nabla_{u}^{2}\widetilde
{V}[(\widetilde{a}\mathbf{;}s_{1})]T_{j}(w)=T_{j}\left(  \frac{1}{s_{1}%
}\widetilde{B}_{j}w\right)  .
\]
The matrices $\widetilde{B}_{j}$ are computed in \cite{GaIz11} and
are given by
\[
\widetilde{B}_{j}=(s_{1}+\widetilde{\alpha}_{j})I-\widetilde{\beta}_{j
}R-\widetilde{\gamma}_{j}iJ
\]
where
\[
\widetilde{\alpha}_{j}=\frac{\alpha-1}{4}(s_{j+1}+s_{j-1}),\qquad
\widetilde{\beta}_{j}=\frac{\alpha+1}{2}(s_{j}-s_{1}),\qquad
\widetilde{\gamma}_{j}=\frac{\alpha-1}{4}(s_{j+1}-s_{j-1}).
\]
The result follows from $\nabla_{u}^{2}V[a]T_{j}(w)=T_{j}(B_{j}w)$ with
$B_{j}=\frac{1}{s_{1}}\widetilde{B}_{j}$.
\end{proof}

We can now find a block diagonalisation of the matrix
\[
M_{a}(\lambda)=\lambda^{2}\mathcal{I}-2i\lambda\mathcal{J}+\nabla_{u}^{2}V[a]
\]
that appears in \eqref{Ma}. With respect to the isotypic decompositions, we get by
\eqref{diaghess} 
\[
P^{-1}M_{a}(\lambda)P=m_{1}(\lambda)\oplus\dots\oplus m_{k}(\lambda),
\]
where each block $m_{j}(\lambda)\in\mbox{End}(E^{\mathbb{C}})$ is given by
\begin{equation}
m_{j}(\lambda)=\lambda^{2}I-2\lambda iJ+B_{j}.\label{Ecblock}%
\end{equation}

\subsubsection*{Zero eigenvalues arising from symmetries}

The analysis of the spectrum of the blocks $m_{j}(\lambda)$ for $j=1,k-1,k$ is
special because the block $m_{k}(\lambda)$ contains the generator of the
$SO(2)$-orbit and the blocks $m_{1}(\lambda)$ and $m_{k-1}(\lambda)$ contain vectors in the
orthogonal complement to $(E_{red}^k)^{\mathbb{C}}$. These cases are treated
separately in Lemma \ref{Prop1} and Lemma \ref{Prop2}.

\begin{lemma}
\label{Prop1} Given the block diagonal matrix $P^{-1}M_{a}(\lambda
)P=m_{1}(\lambda)\oplus\dots\oplus m_{k}(\lambda)$,

\begin{enumerate}
\item The block $m_{k}(\ell/p)$ is invertible for all $\ell\in\mathbb{Z}%
/\{0\}$ when $p\left(  3-\alpha\right)  ^{1/2}\notin\mathbb{N}$.

\item The block $m_{k}(0)\in\mbox{End}(E^{\mathbb{C}})$ is invertible on the
orthogonal complement of the line spanned by $e_{2}=Je_{1}$ which is a
generator of the space tangent to the $SO(2)$-orbit at $a$.
\end{enumerate}
\end{lemma}

\begin{proof}
\begin{enumerate}
\item The coefficients in \eqref{EnBk2} have the property $s_{j}=s_{-j
}=s_{k-j}$ and $s_{k}=0$. The coefficients \eqref{coeff} are then given by
\[
\alpha_{k}=\frac{\alpha-1}{2},\qquad\beta_{k}=-\frac{\alpha+1}{2},\qquad
\gamma_{k}=0.
\]
From \eqref{Bell} we get Thus we have that%
\[
m_{k}(\lambda)=\lambda^{2}I-2\lambda iJ+\frac{\alpha+1}{2}(I+R)=\left(
\begin{array}
[c]{cc}%
\lambda^{2}+\alpha+1 & 2\lambda i\\
-2\lambda i & \lambda^{2}%
\end{array}
\right)  
\]
which is the linearisation of the Kepler problem as in \cite{Braids} and
corresponds to the invariant manifold of homographic solutions. This matrix
has eigenvalues
\[
\mu_{k}^{\pm}(\lambda)=\frac{\alpha+1}{2}+\lambda^{2}\pm\frac{1}{2}%
\sqrt{(\alpha+1)^{2}+16\lambda^{2}}.
\]
The eigenvalue $\mu_{k}^{+}(\lambda)\neq0$ for all $\lambda$ and $\mu_{k}%
^{-}(\lambda)\neq0$ if $\lambda\notin\{0,\sqrt{3-\alpha}\}$. In our analysis,
$\lambda=\ell/p$ where $\ell,p$ are integers and $\ell\neq0$. In particular,
the operator is invertible if we suppose that $p\left(  3-\alpha\right)
^{1/2}\notin\mathbb{N}$. When $\alpha=2$, note that the eigenvalues $\mu
_{k}^{-}(\ell/p)$ for $\ell=\pm p$ are equal to zero due to the existence of
the homographic elliptic orbits of the gravitational $k$-body problem.

\item Since $a=s_{1}^{\frac{1}{\alpha+1}}(\exp(J\zeta)e_{1},\dots,\exp
(kJ\zeta)e_{1})$, the generator of the tangent space of its $SO(2)$-orbit is
\[
\mathcal{J}_{k}a=s_{1}^{\frac{1}{\alpha+1}}(J\exp(J\zeta)e_{1},\dots
,J\exp(kJ\zeta)e_{1})=s_{1}^{\frac{1}{\alpha+1}}T_{k}(e_{2})
\]
since $Je_{1}=e_{2}$. Therefore, the matrix $m_{k}(0)=\mbox{diag}(\alpha+1,0)$
is singular only on the space tangent to the $U(1)$-orbit of $a$ which is
generated by $e_{2}$ in $E^{\mathbb{C}}$.
\end{enumerate}\end{proof}

We now analyze the spectrum of $m_{1}(\lambda)$ and $m_{k-1}(\lambda)$. In our
analysis, the matrix $M_{a}(\lambda)$ is an endomorphism of the subspace
reduced by the symmetry of translations $(E_{red}^{k})^{\mathbb{C}}$.
Therefore, the spectrum of the blocks $m_{1}(\lambda)$ and $m_{k-1}(\lambda)$
must be analyzed on the subspaces $T_{j}^{-1}((W_{j})_{red})$ for $j=1$ and
$j=k-1$ respectively. We use the notation%

\[
(W_{j})_{red}=W_{j}\cap(E_{red}^{k})^{\mathbb{C}}.
\]

\begin{lemma}
\label{Prop2}If $\alpha>1$, the matrix $m_{j}(\lambda)$ restricted to the
subspace ${T_{j}^{-1}((W_{j})_{red})}$ for $j=1,k-1$ is invertible for any
$\lambda$.
\end{lemma}

\begin{proof}
\paragraph{Case $j=1$.}
By \eqref{Ecblock} and Proposition \ref{prop: normal form}, the coefficients
in \eqref{Bell} are
\[
\alpha_{1}=\gamma_{1}=\frac{\alpha-1}{4s_{1}}s_{2}\qquad\beta_{1}=0
\]
and then
\[
m_{1}(\lambda)=(\lambda^{2}+1+\alpha_{1})I-\left(  2\lambda+\alpha_{1}\right)
iJ=\left(
\begin{array}
[c]{cc}%
\lambda^{2}+1+\alpha_{1} & \left(  2\lambda+\alpha_{1}\right)  i\\
-\left(  2\lambda+\alpha_{1}\right)  i & \lambda^{2}+1+\alpha_{1}%
\end{array}
\right)  \text{.}%
\]
The eigenvalues of $m_{1}(\lambda)$ are $\mu_{1}(\lambda)=\left(
\lambda-1\right)  ^{2}\ $with eigenvector $w_{1}=(1,i)$ and $\mu_{2}%
(\lambda)=\left(  \lambda+1\right)  ^{2}+2\alpha_{1}$ with eigenvector
$w_{2}=(1,-i)$. The first eigenvalue $\mu_{1}(\lambda)$ vanishes when
$\lambda=1$, but  we will show that the corresponding eigenvector $w_{1}$ does not belong to
$T_{1}^{-1}((W_{1})_{red})$. 
By \eqref{Tmu}
\[
T_{1}(w_{1})=\frac{1}{\sqrt{k}}(\exp((i I+J)\zeta)w_{1},\dots,\exp(k(i
I+J)\zeta)w_{1}).
\]
Observe that
\[
\exp(j\zeta J)w_{1}=\left(
\begin{array}
[c]{cc}%
\cos j\zeta & -\sin j\zeta\\
\sin j\zeta & \cos j\zeta
\end{array}
\right)  \left(
\begin{array}
[c]{c}%
1\\
i
\end{array}
\right)  =\left(
\begin{array}
[c]{c}%
\cos\left(  j\zeta\right)  -i\sin\left(  j\zeta\right) \\
i\cos\left(  j\zeta\right)  +\sin\left(  j\zeta\right)
\end{array}
\right)  =e^{-ij\zeta}I w_{1}%
\]
for $j=1, \dots, k$ from which it follows that
\[
T_{1}(w_{1})=\frac{1}{\sqrt{k}}\left(  w_{1}, \dots, w_{1}\right) .
\]
In particular, $T_{1}(w_{1})$ does not belong to $(W_{1})_{red}$. The matrix
$m_{1}(\lambda)$ restricted to $T_{1}^{-1}((W_{1})_{red})$ is given by
$\left(  \lambda+1\right)  ^{2}+2\alpha_{1}$, which is invertible for
$\alpha>1$.

\paragraph{Case $j=k-1$.}

By \eqref{Ecblock} and Proposition \ref{prop: normal form}, the coefficients
in \eqref{Bell} are
\[
\alpha_{k-1}=-\gamma_{k-1}=\frac{\alpha-1}{4s_{1}}s_{2}=\alpha_{1}\qquad
\beta_{1}=0
\]
and then
\[
m_{k-1}(\lambda)=\left(
\begin{array}
[c]{cc}%
\lambda^{2}+1+\alpha_{1} & \left(  2\lambda-\alpha_{1}\right)  i\\
-\left(  2\lambda-\alpha_{1}\right)  i & \lambda^{2}+1+\alpha_{1}%
\end{array}
\right)
\]
are $\mu_{1}(\lambda)=\left(  \lambda+1\right)  ^{2}$ with eigenvector
$w_{1}=(1,-i)$ and $\mu_{2}(\lambda)=\left(  \lambda-1\right)  ^{2}%
+2\alpha_{1}$ with eigenvector $w_{2}=(1,i)$. Since $\exp(j\zeta
J)w_{1}=e^{ij\zeta}Iw_{1}$ for $j=1,\dots, k$, we get
\[
T_{k-1}(w_{1})=\frac{1}{\sqrt{k}}(w_{1}, \dots, w_{1}).
\]
In particular, $T_{k-1}(w_{1})$ does not belong to $(W_{k-1})_{red}$. The
matrix $m_{k-1}(\lambda)$ restricted to $T_{k-1}^{-1}((W_{k-1})_{red})$ is
given by $\left(  \lambda-1\right)  ^{2}+2\alpha_{1}$, which is invertible for
$\alpha>1$.
\end{proof}

\subsection{The \texorpdfstring{$2\pi p$}--nondegeneracy property of the \texorpdfstring{$k$}--polygon for weak
forces}

We now study the special case of week forces. 

\begin{proposition}
Assume that  

$$\mbox{(i)}\quad p\left(  3-\alpha\right)  ^{1/2}\notin\mathbb{N}
\qquad \mbox{(ii)}\quad p^{2}\frac{j(k-j)}{k-1}\notin\mathbb{N}\quad j=2,...,k-2.$$

Then there exists $\delta>0$
such that the $k$-polygon is $2\pi p$-nondegenerate for any $\alpha
\in(1,1+\delta)$.
\end{proposition}

\begin{proof}
By Proposition \ref{Prop1} and assumption (i), the block $m_{k}(\ell/p)$ is
invertible for $\ell\neq0$ and $m_{k}(0)$ is invertible in an orthogonal
complement to the generator of the orbit. By Proposition \ref{Prop2} and the
assumption that $\alpha>1$, the blocks $m_{j}(\ell/p)$ restricted to $T_{j}^{-1}((W_{j})_{red})$ are invertible for $j=1,k-1$. It remains to show that the blocks $m_{j}(\ell/p)$ are invertible for all
 $j=2,\dots, k-2$ and $\ell\in\mathbb{Z}$.

In the logarithmic case
$\alpha=1$ the numbers $s_{j}$ can be computed explicitly \cite{GaIz12} and are given by
\[
s_{j}=\frac{j(k-j)}{2}\text{.}%
\]
In this case, the coefficients in \eqref{coeff} are $$\alpha_{j}=\gamma_j=0,\qquad \beta_j=\frac{s_j}{s_1}-1.$$ We can compute the matrix \eqref{Ecblock} explicitly and the determinant is
\[
\det\left( m_{j}(\lambda)\right)=(\lambda^{2}-1)^{2}-\beta_{j}^{2}.
\]
Thus $\det\left( m_{j}(\lambda)\right)\neq 0$ if and only
if
\[
\lambda^{2}\neq\frac{s_{j}}{s_{1}}=\frac{j(k-j)}{k-1}.
\]
Therefore, the blocks $m_{j}(\ell/p)$ are invertible if and only if assumption
(ii) holds. Notice that there is $\ell_0>0$ such that the blocks $m_{j}(\ell/p)$ are always invertible for
$\left\vert \ell\right\vert >\ell_{0}$, which reflects the compactness nature of
the operators. By continuity of $m_{j}(\ell/p)$ with respect to $\alpha$, there is a
$\delta>0$ such that for any $\alpha\in(1,1+\delta)$, the remaining blocks
$m_{j}(\ell/p)$ for $\left\vert \ell\right\vert <\ell_{0}$ are invertible for
all $j=2,\dots,k-2$. The result follows.
\end{proof}

For $\alpha\in(1,1+\delta)$ and $k\in\mathbb{N}$ set
\[
\mathcal{C}_{k,\alpha}=\left\{  p\in\mathbb{N}:p\notin(k-1)\mathbb{N}%
\cup(3-\alpha)^{-1/2}\mathbb{N}\right\}  .
\]
For each prime $k-1$ and $p\in\mathcal{C}_{k,\alpha}$ we have that conditions
(i) and (ii) hold. This follows from the fact that $k-1$ does not divide
$p^{2}$ if $p\in\mathcal{C}_{k,\alpha}$, neither $j(k-j)$ for $j=2,...,k-1$
because $k-1$ is prime.

Now, in the case that $\left(  3-\alpha\right)  ^{-1/2}$ is irrational, the
set $\mathcal{C}_{k,\alpha}$ consists of the integers $p$ that are not divided
by $k-1$. While if $\left(  3-\alpha\right)  ^{-1/2}=\mathfrak{p}%
/\mathfrak{q}$ is rational, then $\left(  3-\alpha\right)  ^{-1/2}>1/2$ for
$\alpha\in(1,1+\delta)$, $\mathfrak{p}>1$ and $\mathcal{C}_{k,\alpha}=\left\{  p\in\mathbb{N}:p\notin(k-1)\mathbb{N}%
\cup\mathfrak{p}\mathbb{N}\right\}  $ is the infinite set of integers $p$ that
are not a multiple of $k-1$ or $\mathfrak{p}$. In both cases $\mathcal{C}%
_{k,\alpha}$ is an infinite set. We have the following theorem,

\begin{theorem}
There is a small $\delta>0$ such that for any $\alpha\in(1,1+\delta)$, the
$k$-polygon is $2\pi p$-nondegenerate if $k-1$ is prime and $p$ is chosen from
the infinite set $\mathcal{C}_{k,\alpha}$.
\end{theorem}

\subsection{The \texorpdfstring{$2\pi/m$}--nondegeneracy property of the \texorpdfstring{$n$}--polygon for
gravitational forces}

In this section we verify the $2\pi/m$-nondegeneracy property in the
gravitational case. Before proving this property, we use computer-assisted
proofs to validate that there are no integers $\ell\in\mathbb{N}$ such that
$\det\left(m_{j}(\ell)\right)=0$ for $j=2, \dots, k-2$.

\begin{proposition}
\label{Prop3}For each $k$ from $4$ to $1000$, the polynomial
\[
P_{j}(\lambda)=\det \left( m_{j}(\lambda)\right)
\]
has no integer roots for $j=2, \dots, k-2$.
\end{proposition}

\begin{proof}
When $\alpha=2$, we get by \eqref{coeff}
$$\alpha_j-\gamma_j=\frac{s_{j-1}}{2s_1}\geq 0\qquad \alpha_j+\gamma_j=\frac{s_{j+1}}{2s_1}\geq 0.$$
Since
\begin{align*}
P_{j}(\lambda)  & =(\left(  \lambda-1\right)  ^{2}+\alpha_{j}-\gamma
_{j})(\left(  \lambda+1\right)  ^{2}+\alpha_{j}+\gamma_{j})-\beta_{j}^{2}\\
& \geq\left(  \lambda-1\right)  ^{2}\left(  \lambda+1\right)  ^{2}-\beta
_{j}^{2}\geq\left(  \lambda^{2}-1\right)  ^{2}-\beta_{j}^{2}%
\end{align*}
The polynomial has not roots for $\lambda^{2}-1\geq\beta_{j}^{2}$. The result
is obtained by validating rigorously, using interval arithmetics in the package INTLAB in
MATLAB, that $P_{j}(\ell)\neq0$ for $\ell=0,\dots,\sqrt{\beta_{j}^{2}+1}$ and
$j=2, \dots, k-2$.
\end{proof}

\begin{theorem}
For $\alpha=2$, the $k$-polygon is $2\pi/m$-nondegenerate with $m>1$ for any
$k=4,\dots,1000$. In addition, there is a small $\delta>0$ such that for all
$\alpha\in(2-\delta,2+\delta)/\{2\}$, the $k$-polygon is $2\pi$-nondegenerate
for any $k=4,\dots,1000$.
\end{theorem}

\begin{proof}
By Proposition \ref{Prop1} and the fact that $\sqrt{3-\alpha}$ is integer only
if $\alpha=2$ with $\sqrt{3-\alpha}=1$, we have that the block $m_{k}(\ell)$
is fine for $\ell\in m\mathbb{Z}$ if $\alpha=2$ and for $\ell\in\mathbb{Z}$ if
$\alpha\in(2-\delta,2+\delta)\setminus \{2\}$. The blocks $m_{j}(\ell
)$ restricted to ${T_{j}^{-1}((W_{j})_{red})}$ are always invertible for
$j=1,k-1$ by Proposition \ref{Prop2}. Proposition \ref{Prop3} and the
continuity respect $\alpha$ imply that the blocks $m_{j}(\ell)$ are invertible
for all $\alpha\in(2-\delta,2+\delta)$, $j=2,\dots,n-2$ and $\ell\in
\mathbb{Z}$. The result follows.
\end{proof}

\subsection{The \texorpdfstring{$2\pi p$}--nondegeneracy property of the Lagrange triangle for
different masses}

Since the conditions of Remark \ref{Rem} for the Hamiltonian is equivalent
to the $2\pi p$--nondegeneracy property, we can use previous computations
made for the analysis of the stability of the Lagrange triangular configuration.
In the case $\alpha =2$, according to the Gascheau result regarding the
$3$-body problem \cite{roberts}, the linear hamiltonian system at the
Lagrange triangular configuration with masses $m_{j}$ has four pairs of
zero-eigenvalues and four complex roots off the imaginary axis (leading to
instability) when 
\begin{equation*}
\beta =27\frac{m_{1}m_{2}+m_{1}m_{3}+m_{2}m_{3}}{\left(
m_{1}+m_{2}+m_{3}\right) ^{2}}>1\text{.}
\end{equation*}%
The four pairs of zero-eigenvalues correspond to the center of mass and the
homographic elliptic orbits. The linearization of the homographic elliptic
orbits leads to our block $m_{k}(\lambda )$. Therefore, the Lagrange
triangular configuration is $2\pi /m$-nondegenerate for any $m>1$ when $%
\beta >1$. 

In the case $\alpha \neq 2$, after fixing the center of mass equal to zero,
there are four pairs of non-zero eigenvalues. Two pairs correspond to the
linearization in the submanifold of homographic solutions having the block
of the Kepler problem in our analysis $m_{k}(\lambda )$. This block is $2\pi
p$-nonresonant when $p\notin (3-\alpha )^{-1/2}\mathbb{N}$. The other two
pairs of eigenvalues can be found explicitly in \cite{Port}, and correspond
in our setting to
\begin{eqnarray*}
\lambda _{1}^{\pm } &=&\pm \frac{1}{6}i\sqrt{18(1-\alpha )+6\sqrt{9(\alpha
-1)^{2}-\beta (\alpha +3)^{2}}} \\
\lambda _{2}^{\pm } &=&\pm \frac{1}{6}i\sqrt{18(1-\alpha )-6\sqrt{9(\alpha
-1)^{2}-\beta (\alpha +3)^{2}}}.
\end{eqnarray*}%
These eigenvalues are off the imaginary axis when 
\begin{equation}
\beta >9\left( \frac{3-\alpha }{1+\alpha }\right) ^{2}.  \label{routh}
\end{equation}%
Therefore, we conclude that for $\alpha \neq 2$, the Lagrange triangular
configuration is $2\pi p$-nondegenerate when the inequalities $p\notin
(3-\alpha )^{-1/2}\mathbb{N}$ and (\ref{routh}) hold. Furthermore, since the
eigenvalues are analytic in $\alpha $ and $\beta $, the Lagrange triangular
configuration is generically $2\pi p$-nondegenerate for the homogeneous
exponent $\alpha $ and the set of masses $m_{j}$. Observe that for equal
masses $m_{j}=1$ the inequality (\ref{routh}) does not hold precisely for $%
\alpha =1$. This is the degeneracy found in the previous analysis for the
blocks $m_{1}$ and $m_{n-1}$ in the case of $\alpha =1$. 

\paragraph{Acknowledgements.}

We acknowledge the assistance of Ramiro Chavez Tovar with the preparation of
the figures. M. Fontaine is
supported by the FWO-EoS project G0H4518N. C. Garc\'{\i}a-Azpeitia is
supported by PAPIIT-UNAM grant  IA100121.

\vspace{0.5cm} 
\begin{minipage}[t]{7cm}
MF:  {marine.fontaine.math@gmail.com}\\
{\tt Departement Wiskunde\\
Universiteit Antwerpen \\
2020 Antwerpen, BE.}

\end{minipage}
\hfill 
\begin{minipage}[t]{10cm}
CGA:  {cgazpe@mym.iimas.unam.mx}\\
{\tt Depto. Matem\'{a}ticas y Mec\'{a}nica IIMAS \\
Universidad Aut\'onoma de M\'exico \\
Apdo. Postal 20-726, Ciudad de M\'exico, MX.}
\end{minipage}
\end{document}